\documentclass[12pt]{article}%
\usepackage{amsfonts}
\usepackage{amsmath}
\usepackage{lineno}
\usepackage{hyperref}
\usepackage{amssymb}
\usepackage{graphicx}%
\providecommand{\U}[1]{\protect\rule{.1in}{.1in}}
\newtheorem{theorem}{Theorem}[section]

\numberwithin{equation}{section}
\newenvironment{proof}[1][Proof]{\noindent\textbf{#1.} }{\ \rule{0.5em}{0.5em}}
\begin{document}

\begin{center}
\textbf{The Marshall-Olkin extended generalized Gompertz distribution\medskip
}\bigskip

\textbf{Lazhar Benkhelifa}\medskip\medskip\medskip

\textit{Laboratory of Applied Mathematics, Mohamed Khider University, Biskra,
07000, Algeria\medskip}

\textit{Departement of Mathematics and \textit{I}nformatics, Larbi Ben M'Hidi
University, Oum El Bouaghi, 04000, Algeria\bigskip}

l.benkhelifa@yahoo.fr
\end{center}

\noindent\textbf{Abstract}\medskip

\noindent A new four-parameter model called the Marshall-Olkin extended
generalized Gompertz distribution is introduced. Its hazard rate function can
be constant, increasing, decreasing, upside-down bathtub or bathtub-shaped
depending on its parameters. Some mathematical properties of this model such
as expansion for the density function, moments, moment generating function,
quantile function, mean deviations, mean residual life, order statistics and
R\'{e}nyi entropy are derived. The maximum likelihood technique is used to
estimate the unknown model parameters and the observed information matrix is
determined. The applicability of the proposed model is shown by means of a
real data set.\bigskip\bigskip\bigskip

\bigskip\bigskip

\section{\textbf{Introduction\label{sec 1}}}

The Gompertz distribution which was proposed by Gompertz in 1825 \cite{r11}
plays an important role in modeling survival times, human mortality and
actuarial data. It is applied in several areas such as biology, gerontology,
computer and marketing science, among others. Some applications of the
Gompertz distribution can be found in \cite{r13}. In recent years, some
extensions of the Gompertz distribution have been proposed in the literature.
El-Gohary et al. \cite{r6} introduced the generalized Gompertz (GG)
distribution using the Lehmann alternative type I \cite{r14} with cumulative
distribution function (cdf)%
\begin{equation}
G\left(  x\right)  =\left(  1-e^{-\frac{\beta}{\lambda}\left(  e^{\lambda
x}-1\right)  }\right)  ^{\alpha},\ x,\alpha,\beta,\lambda>0. \label{1}%
\end{equation}
Jafari et al. \cite{r12} defined the beta Gompertz distribution, using the
beta-G generators defined by Eugene et al. (2002), with cdf%
\[
G\left(  x\right)  =\frac{1}{B\left(  a,b\right)  }\int_{0}^{1-e^{-\frac
{\beta}{\lambda}\left(  e^{\lambda x}-1\right)  }}t^{a-1}\left(  1-t\right)
^{b-1}dt,\text{ }x,\beta,\lambda,a,b>0.
\]
Roozegar et al. \cite{r18} introduced the McDonald Gompertz distribution,
using the\ McDonald-G generators defined by Alexander et al. \cite{r1}, with
cdf%
\[
G\left(  x\right)  =\frac{1}{B\left(  a/c,b\right)  }\int_{0}^{\left(
1-e^{-\frac{\beta}{\lambda}\left(  e^{\lambda x}-1\right)  }\right)  ^{c}%
}t^{a/c-1}\left(  1-t\right)  ^{b-1}dt,\text{ }x,\beta,\lambda,a,b,c>0\text{.}%
\]
\medskip

\noindent A new method for adding a parameter to a family of distributions is
proposed by Marshall and Olkin \cite{r15}. The resulting distribution, known
as Marshall-Olkin extended (denoted by the prefix "MOE" for short)
distribution. The MOE-G distribution gives more flexibility for modeling
various types of real data in practice. For any baseline cdf $G\left(
x\right)  ,$ $x\in%
\mathbb{R}
,$ the cdf of the MOE-G distribution is given by%
\begin{equation}
F\left(  x\right)  =\frac{G\left(  x\right)  }{\theta+\overline{\theta
}G\left(  x\right)  },\text{ }\theta>0, \label{2}%
\end{equation}
where $\overline{\theta}=1-\theta$ is\textbf{\ }a tilt parameter. The
probability density function (pdf) of the MOE-G distribution is%
\[
f\left(  x\right)  =\frac{\theta g\left(  x\right)  }{\left[  \theta
+\overline{\theta}G\left(  x\right)  \right]  ^{2}},
\]
where $g\left(  x\right)  =dG\left(  x\right)  /dx$ is the pdf of $G\left(
x\right)  .$ The MOE-G distribution becomes the baseline distribution when
$\theta=1.$\smallskip

\noindent Several new distributions have been proposed by using the
Marshall-Olkin method. For example, the MOE Pareto distribution has been
introduced by Alice and Jose \cite{r2}, the MOE Lomax distribution has been
introduced by Ghitany et al. \cite{r9}, the MOE gamma distribution has been
proposed by Rist\'{\i}c et al. \cite{r19}, the MOE normal distribution has
been proposed by Garcia et al. \cite{r8}, Rist\'{\i}c and Kundu \cite{r20}
defined the MOE generalized exponential distribution, Okasha et al. \cite{r17}
introduced the MOE generalized linear exponential distribution and Benkhelifa
\cite{r3} introduced the MOE generalized Lindley distribution. Cordeiro et al.
\cite{r4} studied the general properties of the MOE-G distribution.\medskip

\noindent In this paper, we propose a new four-parameter continuous model,
so-called the MOEGG distribution. We also provide several mathematical
properties of the proposed model and apply it to real data.\medskip

\noindent The remainder of this paper is organized as follows. In Section
\ref{sec 2}, we define the MOEGG distribution and present some of its special
cases. Some mathematical properties of the new model are given in Section
\ref{sec 3}.\ In Section \ref{sec 4}, we estimate the model parameters by the
maximum likelihood method and calculate the elements of the observed
information matrix. In Section \ref{sec 5}, we illustrate the flexibility of
the new distribution by using a real data set. Finally, we give some
conclusions in Section \ref{sec 6}.

\section{\textbf{The MOEGG distribution\label{sec 2}}}

By inserting (\ref{1}) into (\ref{2}), we get the cdf of the MOEGG
distribution that is%
\begin{equation}
F\left(  x\right)  =\frac{\left(  1-e^{-\frac{\beta}{\lambda}\left(
e^{\lambda x}-1\right)  }\right)  ^{\alpha}}{\theta+\overline{\theta}\left(
1-e^{-\frac{\beta}{\lambda}\left(  e^{\lambda x}-1\right)  }\right)  ^{\alpha
}}. \label{3}%
\end{equation}
The pdf corresponding to (\ref{3}) is%
\begin{equation}
f\left(  x\right)  =\frac{\alpha\beta\theta e^{\lambda x}e^{-\frac{\beta
}{\lambda}\left(  e^{\lambda x}-1\right)  }\left(  1-e^{-\frac{\beta}{\lambda
}\left(  e^{\lambda x}-1\right)  }\right)  ^{\alpha-1}}{\left[  \theta
+\overline{\theta}\left(  1-e^{-\frac{\beta}{\lambda}\left(  e^{\lambda
x}-1\right)  }\right)  ^{\alpha}\right]  ^{2}}. \label{4}%
\end{equation}
Hereafter, when $X$ is a random variable following the MOEGG distribution, it
will be denoted by $X\sim MOEGG\left(  \alpha,\beta,\lambda,\theta\right)
$.\ Plots of the MOEGG density for selected parameter values are displayed in
Figure 1. We observe that the density function can take various forms,
depending on the parameter values. It is evident that the MOEGG distribution
is much more flexible than the GG distribution.\smallskip

\noindent The hazard rate function defined by $h(x)=f(x)/[1-F(x)]$ is an
important quantity characterizing lifetime phenomena. It can be loosely
interpreted as the conditional probability of failure, given that it has
survived to time $t$. The hazard rate function for the MOEGG random variable
is given by%
\begin{equation}
h\left(  x\right)  =\frac{\alpha\beta e^{\lambda x}e^{-\frac{\beta}{\lambda
}\left(  e^{\lambda x}-1\right)  }\left(  1-e^{-\frac{\beta}{\lambda}\left(
e^{\lambda x}-1\right)  }\right)  ^{\alpha-1}}{\left[  1-\left(
1-e^{-\frac{\beta}{\lambda}\left(  e^{\lambda x}-1\right)  }\right)  ^{\alpha
}\right]  \left[  \theta+\overline{\theta}\left(  1-e^{-\frac{\beta}{\lambda
}\left(  e^{\lambda x}-1\right)  }\right)  ^{\alpha}\right]  }. \label{5}%
\end{equation}

\noindent Plots of the hazard rate function of the MOEGG distribution for
different values are given in Figure 2. We observe that the hazard rate
functions of the MOEGG distribution can be constant, increasing, decreasing,
upside-down bathtub and bathtub-shaped depending basically on the values of
the parameters.\medskip

\noindent The MOEGG distribution includes as special sub-models the following distributions:

\begin{itemize}
\item If $\theta=1$, then we obtain the GG distribution.

\item If $\theta=1\ $and $\lambda$ tends to zero, then we get the generalized
exponential distribution.

\item If $\theta=1$ and\ $\alpha=1,$ then we get the Gompertz distribution.

\item If $\theta=1,$\ $\alpha=1,$ and $\lambda$ tends to zero, then we get the
exponential distribution.

\item If $\alpha=1,$ then we get the MOE Gompertz distribution.

\item If $\lambda$ tends to zero, then we get the MOE generalized exponential distribution.

\item If $\alpha=1$ and $\lambda$ tends to zero$,$ then we get the MOE
exponential distribution.
\end{itemize}

%

{\parbox[b]{6.0995in}{\begin{center}
\includegraphics[
height=2.9585in,
width=6.0995in
]%
{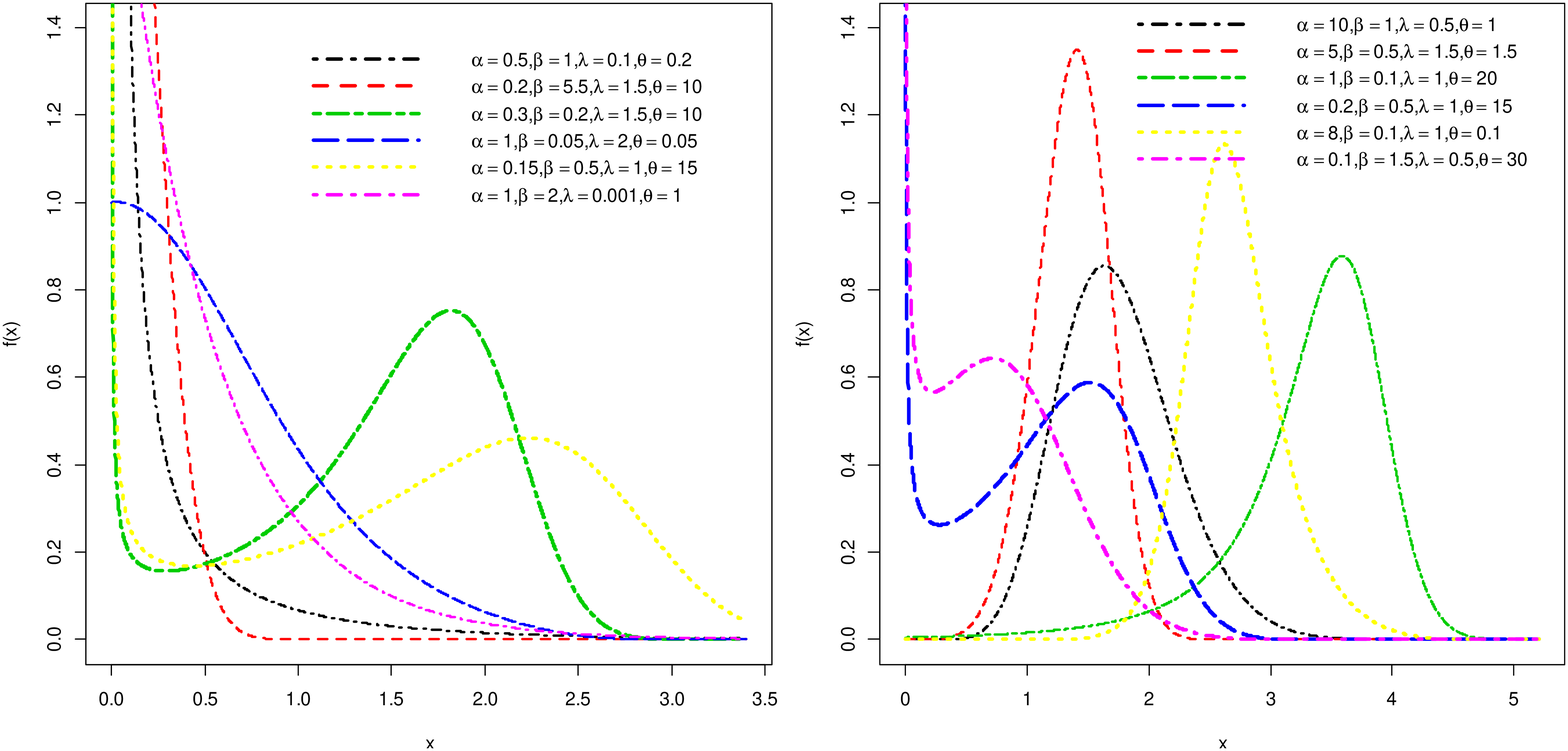}%
\\
Figure 1. The MOEGG density function for some values of $\alpha,\beta
,\lambda\ $and $\theta.$%
\end{center}}}
%

{\parbox[b]{6.0995in}{\begin{center}
\includegraphics[
height=2.9585in,
width=6.0995in
]%
{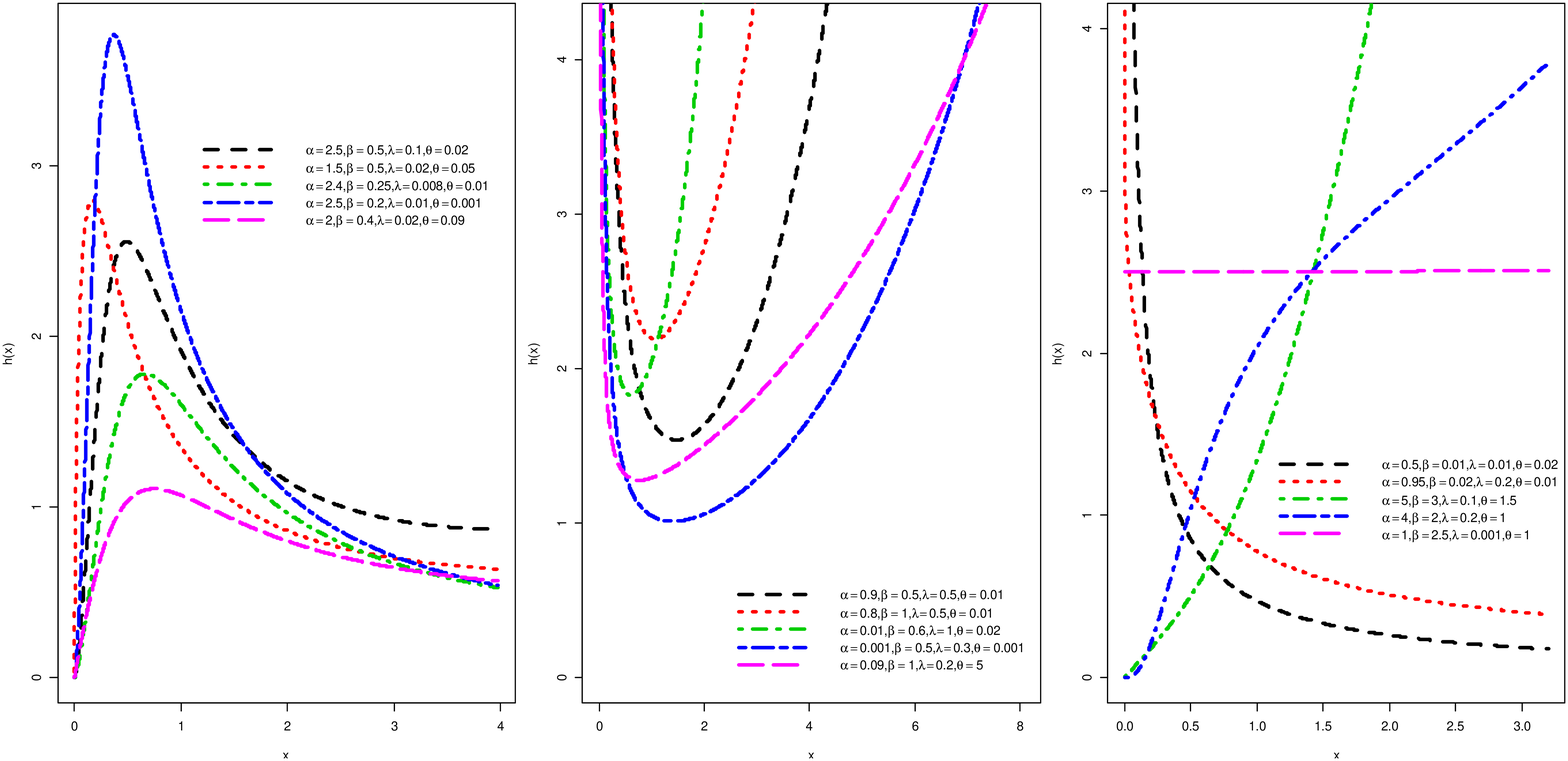}%
\\
Figure 2. The MOEGG hazard rate function for some values of $\alpha
,\beta,\lambda\ $and $\theta.$%
\end{center}}}

\section{\textbf{Mathematical properties\label{sec 3}}}

In this section, we study some mathematical properties of the MOEGG model.

\subsection{\textit{Useful expansions}}

In this subsection, we give useful expansions for the pdf and cdf of the MOEGG
distribution. If $|z|<1$ and $\delta>0$ is a real non-integer, we have the
expansion%
\begin{equation}
\left(  1-z\right)  ^{-\delta}=\sum_{j=0}^{\infty}\binom{\delta+j-1}{j}z^{j},
\label{6}%
\end{equation}
If $\delta$ is an integer, the index $j$ in the previous sum stops at
$\delta-1$. We can rewrite, the pdf of the MOEGG distribution (\ref{4}) as
\[
f\left(  x\right)  =\frac{\alpha\beta\theta e^{\lambda x}e^{-\frac{\beta
}{\lambda}\left(  e^{\lambda x}-1\right)  }\left(  1-e^{-\frac{\beta}{\lambda
}\left(  e^{\lambda x}-1\right)  }\right)  ^{\alpha-1}}{\left[  1-\overline
{\theta}\left(  1-\left(  1-e^{-\frac{\beta}{\lambda}\left(  e^{\lambda
x}-1\right)  }\right)  ^{\alpha}\right)  \right]  ^{2}}.
\]
Then, if $\ \theta\in\left(  0,1\right)  ,$ by applying the expansion
(\ref{6}), we get%
\begin{align*}
\left[  1-\overline{\theta}\left(  1-\left(  1-e^{-\frac{\beta}{\lambda
}\left(  e^{\lambda x}-1\right)  }\right)  ^{\alpha}\right)  \right]  ^{-2}
&  =\sum_{j=0}^{\infty}\left(  j+1\right)  \overline{\theta}^{j}\left(
1-\left(  1-e^{-\frac{\beta}{\lambda}\left(  e^{\lambda x}-1\right)  }\right)
^{\alpha}\right)  ^{j}\\
&  =\sum_{j=0}^{\infty}\left(  j+1\right)  \overline{\theta}^{j}\sum_{i=0}%
^{j}\left(  -1\right)  ^{i}\binom{j}{i}\left(  1-e^{-\frac{\beta}{\lambda
}\left(  e^{\lambda x}-1\right)  }\right)  ^{\alpha i}.
\end{align*}
Therefore%
\[
f\left(  x\right)  =\alpha\beta\theta e^{\lambda x}e^{-\frac{\beta}{\lambda
}\left(  e^{\lambda x}-1\right)  }\sum_{j=0}^{\infty}\left(  j+1\right)
\overline{\theta}^{j}\sum_{i=0}^{j}\left(  -1\right)  ^{i}\binom{j}{i}\left(
1-e^{-\frac{\beta}{\lambda}\left(  e^{\lambda x}-1\right)  }\right)
^{\alpha\left(  i+1\right)  -1}.
\]
By interchanging $\sum_{j=0}^{\infty}\sum_{i=0}^{j}$ with $\sum_{i=0}^{\infty
}\sum_{j=i}^{\infty}$ in the last equation, we get%
\begin{equation}
f\left(  x\right)  =\sum_{i=0}^{\infty}w_{i}g\left(  x;\alpha\left(
i+1\right)  ,\beta,\lambda\right)  , \label{7}%
\end{equation}
\bigskip where%
\[
w_{i}=\frac{\theta\left(  -1\right)  ^{i}}{\left(  i+1\right)  }\sum
_{j=i}^{\infty}\binom{j}{i}\left(  j+1\right)  \overline{\theta}^{j},\text{
for }i=0,1,\ldots\text{.}%
\]
By integrating (\ref{7}), we obtain%
\[
F\left(  x\right)  =\sum_{i=0}^{\infty}w_{i}G\left(  x;\alpha\left(
i+1\right)  ,\beta,\lambda\right)  .
\]
\bigskip

\noindent Otherwise, if $\theta>1,$ we can rewrite the MOEGG density function
(\ref{4}) as%
\[
f\left(  x\right)  =\frac{\alpha\beta e^{\lambda x}e^{-\frac{\beta}{\lambda
}\left(  e^{\lambda x}-1\right)  }\left(  1-e^{-\frac{\beta}{\lambda}\left(
e^{\lambda x}-1\right)  }\right)  ^{\alpha-1}}{\theta\left[  1-\left(
1-\theta^{-1}\right)  \left(  1-e^{-\frac{\beta}{\lambda}\left(  e^{\lambda
x}-1\right)  }\right)  ^{\alpha}\right]  ^{2}},
\]
then by using the expansion (\ref{6}), we get
\begin{equation}
f\left(  x\right)  =\sum_{i=0}^{\infty}v_{i}g\left(  x;\alpha\left(
i+1\right)  ,\beta,\lambda\right)  , \label{8}%
\end{equation}
where $v_{i}=\theta^{-1}\left(  1-\theta^{-1}\right)  ^{i}$ for $i=0,1,\ldots
$. By integrating (\ref{8}), we obtain%
\[
F\left(  x\right)  =\sum_{i=0}^{\infty}v_{i}G\left(  x;\alpha\left(
i+1\right)  ,\beta,\lambda\right)  .
\]
\medskip

\noindent It easy to verify that $\sum_{i=0}^{\infty}v_{i}=\sum_{i=0}^{\infty
}w_{i}=1.$ Equations (\ref{7}) and (\ref{8}) show that the MOEGG density
function (\ref{4}) is an infinite linear combination of GG density functions
and have the same form except for the coefficients which are $w_{i}^{\prime}s$
in (\ref{7}) and $v_{i}^{\prime}s$ in (\ref{8}). Hence, several mathematical
properties of the MOEGG distribution can be obtained directly from the
properties of the GG distribution.\bigskip

\subsection{\textit{Moments}}

Let $X\sim MOEGG\left(  \alpha,\beta,\lambda,\theta\right)  $. The $r$th
moment of $X,$\ say $E\left(  X^{r}\right)  ,$ is given by%
\[
E\left(  X^{r}\right)  =\int_{0}^{\infty}x^{r}f\left(  x\right)  dx.
\]
From (\ref{8}), if $\theta>1$, we have%
\[
E\left(  X^{r}\right)  =\sum_{i=0}^{\infty}v_{i}E\left(  X_{i+1}^{r}\right)
,
\]
where $X_{i+1}$ denotes a random variable having the GG density function
$g\left(  x;\alpha\left(  i+1\right)  ,\beta,\lambda\right)  .$ Then, we have%
\[
E\left(  X_{i+1}^{r}\right)  =\alpha\left(  i+1\right)  \beta\int_{0}^{\infty
}x^{r}e^{\lambda x}e^{-\frac{\beta}{\lambda}\left(  e^{\lambda x}-1\right)
}\left(  1-e^{-\frac{\beta}{\lambda}\left(  e^{\lambda x}-1\right)  }\right)
^{\alpha\left(  i+1\right)  -1}dx.
\]
Since $0<e^{-\frac{\beta}{\lambda}\left(  e^{\lambda x}-1\right)  }<1$ for
$x>0$, by applying the binomial expansion of $\left(  1-e^{-\frac{\beta
}{\lambda}\left(  e^{\lambda x}-1\right)  }\right)  ^{\alpha\left(
i+1\right)  -1}$ given by%
\[
\left(  1-e^{-\frac{\beta}{\lambda}\left(  e^{\lambda x}-1\right)  }\right)
^{\alpha\left(  i+1\right)  -1}=\sum_{j=0}^{\infty}\binom{\alpha\left(
i+1\right)  -1}{j}\left(  -1\right)  ^{j}e^{-\frac{j\beta}{\lambda}\left(
e^{\lambda x}-1\right)  },
\]
we obtain%
\[
E\left(  X_{i+1}^{r}\right)  =\alpha\left(  i+1\right)  \beta\sum
_{j=0}^{\infty}\binom{\alpha\left(  i+1\right)  -1}{j}\left(  -1\right)
^{j}e^{\frac{\beta}{\lambda}\left(  j+1\right)  }\int_{0}^{\infty}%
x^{r}e^{\lambda x}e^{-\frac{\beta}{\lambda}\left(  j+1\right)  e^{\lambda x}%
}dx;
\]
By substituting $z=e^{\lambda x}$ in the above integral, we get%
\[
E\left(  X_{i+1}^{r}\right)  =\alpha\left(  i+1\right)  \beta\sum
_{j=0}^{\infty}\binom{\alpha\left(  i+1\right)  -1}{j}\frac{\left(  -1\right)
^{j}e^{\frac{\beta}{\lambda}\left(  j+1\right)  }}{\lambda^{r+1}}\int%
_{1}^{\infty}\left(  \ln z\right)  ^{r}e^{-\frac{\beta}{\lambda}\left(
j+1\right)  z}dz.
\]
By integration by parts, on the above integral, we get
\[
E\left(  X_{i+1}^{r}\right)  =\frac{\alpha\left(  i+1\right)  r!}{\lambda^{r}%
}\sum_{j=0}^{\infty}\binom{\alpha\left(  i+1\right)  -1}{j}\frac{\left(
-1\right)  ^{j}e^{\frac{\beta}{\lambda}\left(  j+1\right)  }}{\left(
j+1\right)  }E_{1}^{r-1}\left(  \frac{\beta}{\lambda}\left(  j+1\right)
\right)  ,
\]
where $E_{s}^{k}\left(  z\right)  =\frac{1}{k!}\int_{1}^{\infty}\left(  \ln
x\right)  ^{k}x^{-s}e^{-zx}dz\ $is the generalized integro-exponential
function (see \cite{r11}).\bigskip

\noindent Therefore%
\[
E\left(  X^{r}\right)  =\sum_{i,j=0}^{\infty}\binom{\alpha\left(  i+1\right)
-1}{j}\frac{\left(  -1\right)  ^{j}r!\alpha\left(  i+1\right)  v_{i}%
e^{\frac{\beta}{\lambda}\left(  j+1\right)  }}{\left(  j+1\right)  \lambda
^{r}}E_{1}^{r-1}\left(  \frac{\beta}{\lambda}\left(  j+1\right)  \right)  .
\]

\bigskip

\noindent Similary, if $\theta\in\left(  0,1\right)  $, the $r$th moment of
$X$ is obtained, from (\ref{7}), by replacing $w_{j}$ with $v_{j}$ in the
above equation.\bigskip

\noindent The first four moments can be used to describe any data. The first
moment or the mean is a measure of the center of the distribution. The second
moment about the mean is equal to the variance which measures the spread of
the distribution about the mean. The skewness measures the symmetry of a
distribution while the kurtosis measures the peakedness of a distribution. The
skewness and the kurtosis, respectively, are%
\[
\gamma_{1}=\frac{\mu_{3}}{\mu_{2}^{3/2}}\text{ \ and \ }\gamma_{2}=\frac
{\mu_{4}}{\mu_{2}^{2}},
\]
where $\mu_{2},$ $\mu_{3}$ and $\mu_{4}$ are the second, third and fourth
central moments. The $p$th central moments, denoted by $\mu_{p},$ or the $p$th
moments about the mean are expressed in terms of moments. By definition%
\begin{align*}
\mu_{p}  &  =E\left(  X-E\left(  X\right)  \right)  ^{p}\\
&  =\sum_{m=0}^{p}\frac{\left(  -1\right)  ^{m}p!}{m!\left(  p-m\right)
!}\left(  E\left(  X\right)  \right)  ^{m}E\left(  X^{p-m}\right)  .
\end{align*}
Table 1 lists the first four non-central moments, variance, skewness and
kurtosis for selected values of the parameter $\theta$ of the MOEGG
distribution for $\alpha=0.5,\beta=1$ and $\lambda=2.$

\begin{center}%
\begin{table}[h] \centering
\caption{Moments, variance, skewness and kurtosis of the MOEGG distribution.}%
\begin{tabular}
[c]{lcccc}\hline
$E\left(  X^{r}\right)  $ & $\theta=1.5$ & $\theta=5$ & $\theta=0.5$ &
$\theta=0.1$\\\hline
$E\left(  X\right)  $ & $0.37187$ & $0.58696$ & $0.20827$ &
\multicolumn{1}{l}{$0.06695$}\\
$E\left(  X^{2}\right)  $ & $0.22837$ & $0.44545$ & $0.10431$ &
\multicolumn{1}{l}{$0.02598$}\\
$E\left(  X^{3}\right)  $ & $0.16956$ & $0.37802$ & $0.06976$ &
\multicolumn{1}{l}{$0.01583$}\\
$E\left(  X^{4}\right)  $ & $0.14115$ & $0.34539$ & $0.05456$ &
\multicolumn{1}{l}{$0.01183$}\\
Variance & $0.09008$ & $0.10093$ & $0.06094$ & \multicolumn{1}{l}{$0.021498$%
}\\
Skewness & $0.65217$ & $-0.06003$ & $1.50583$ & \multicolumn{1}{l}{$121.95925
$}\\
Kurtosis & $-0.40545$ & $-0.78352$ & $89.26947$ &
\multicolumn{1}{l}{$831.79977$}\\\hline
\end{tabular}%
\end{table}%

\end{center}

\subsection{\textit{Moment generating function}}

Let $X\sim MOEGG\left(  \alpha,\beta,\lambda,\theta\right)  .$ The moment
generating function $M\left(  t\right)  =E\left(  e^{tX}\right)  $ of $X$ is%
\[
M\left(  t\right)  =\int_{0}^{\infty}e^{tx}f\left(  x\right)  dx.
\]
In this Subsection, we give two representations for $M\left(  t\right)  $ only
for the case\ $\theta>1,$ since the case $\theta\in\left(  0,1\right)  $ is
completely analogous.\bigskip

\noindent A first representation, using the Maclaurin series expansion of an
exponential function, is%
\[
M\left(  t\right)  =\sum_{r=0}^{\infty}\frac{t^{r}}{r!}E\left(  X^{r}\right)
,
\]
where
\[
E\left(  X^{r}\right)  =\sum_{i,j=0}^{\infty}\binom{\alpha\left(  i+1\right)
-1}{j}\frac{\left(  -1\right)  ^{j}v_{i}\alpha\left(  i+1\right)
r!e^{\frac{\beta}{\lambda}\left(  j+1\right)  }}{\left(  j+1\right)
\lambda^{r}}E_{1}^{r-1}\left(  \frac{\beta}{\lambda}\left(  j+1\right)
\right)  .
\]
Then%
\[
M\left(  t\right)  =\sum_{i,j,r=0}^{\infty}\binom{\alpha\left(  i+1\right)
-1}{j}\frac{\left(  -1\right)  ^{j}v_{i}\alpha\left(  i+1\right)
t^{r}e^{\frac{\beta}{\lambda}\left(  j+1\right)  }}{\left(  j+1\right)
\lambda^{r}}E_{1}^{r-1}\left(  \frac{\beta}{\lambda}\left(  j+1\right)
\right)  .
\]
\bigskip A second representation for $M\left(  t\right)  $ follows from
(\ref{8}) as%
\[
M\left(  t\right)  =\sum_{i=0}^{\infty}v_{i}M_{i+1}\left(  t\right)  ,
\]
where%
\[
M_{i+1}\left(  t\right)  =\alpha\left(  i+1\right)  \beta\int_{0}^{\infty
}e^{tx}e^{\lambda x}e^{-\frac{\beta}{\lambda}\left(  e^{\lambda x}-1\right)
}\left(  1-e^{-\frac{\beta}{\lambda}\left(  e^{\lambda x}-1\right)  }\right)
^{\alpha\left(  i+1\right)  -1}dx.
\]
By using the binomial series expansion, we get%
\[
M_{i+1}\left(  t\right)  =\alpha\left(  i+1\right)  \beta\sum_{j=0}^{\infty
}\binom{\alpha\left(  i+1\right)  -1}{j}\left(  -1\right)  ^{j}e^{\frac{\beta
}{\lambda}\left(  j+1\right)  }\int_{0}^{\infty}e^{tx}e^{\lambda x}%
e^{-\frac{\beta}{\lambda}\left(  j+1\right)  e^{\lambda x}}dx.
\]
The change of variable $z=e^{\lambda x}$ yields%
\[
M_{i+1}\left(  t\right)  =\frac{\alpha\left(  i+1\right)  \beta}{\lambda}%
\sum_{j=0}^{\infty}\binom{\alpha\left(  i+1\right)  -1}{j}\left(  -1\right)
^{j}e^{\frac{\beta}{\lambda}\left(  j+1\right)  }\int_{1}^{\infty}z^{\frac
{t}{\lambda}}e^{-\frac{\beta}{\lambda}\left(  j+1\right)  z}dz.
\]
Then, by setting $y=\frac{\beta}{\lambda}\left(  j+1\right)  z$, we obtain%
\[
M_{i+1}\left(  t\right)  =\alpha\left(  i+1\right)  \left(  \frac{\lambda
}{\beta}\right)  ^{\frac{t}{\lambda}+1}\sum_{j=0}^{\infty}\binom{\alpha\left(
i+1\right)  -1}{j}\frac{\left(  -1\right)  ^{j}e^{\frac{\beta}{\lambda}\left(
j+1\right)  }}{\left(  j+1\right)  ^{\frac{t}{\lambda}+1}}\Gamma\left(
\frac{t}{\lambda}+1,\frac{\beta}{\lambda}\left(  j+1\right)  \right)  ,
\]
where $\Gamma\left(  u,v\right)  $ is the incomplete gamma function defined by%
\[
\Gamma\left(  u,v\right)  =\int_{v}^{\infty}x^{u-1}e^{-x}dx.
\]
Thererfore$,$ if $\theta>1,$ we have
\[
M\left(  t\right)  =\sum_{i,j=0}^{\infty}\binom{\alpha\left(  i+1\right)
-1}{j}\left(  \frac{\lambda}{\beta}\right)  ^{\frac{t}{\lambda}+1}\frac
{v_{i}\alpha\left(  i+1\right)  \left(  -1\right)  ^{j}e^{\frac{\beta}%
{\lambda}\left(  j+1\right)  }}{\left(  j+1\right)  ^{\frac{t}{\lambda}+1}%
}\Gamma\left(  \frac{t}{\lambda}+1,\frac{\beta}{\lambda}\left(  j+1\right)
\right)  .
\]

\subsection{\textit{Quantile function}}

The quantile function of the MOEGG distribution is given by%
\[
Q(u)=F^{-1}\left(  u\right)  =\frac{1}{\lambda}\ln\left(  1-\frac{\beta
}{\lambda}\ln\left(  1-\left(  \frac{\theta u}{1-\overline{\theta}u}\right)
^{\alpha}\right)  \right)  ,\text{ }u\in\left(  0,1\right)  ,
\]
where $F^{-1}\left(  .\right)  $ is the inverse distribution function. It is
easy to simulate the MOEGG distribution. Let $U$ be a continuous uniform
variable on the unit interval $\left(  0,1\right)  $. Thus, using the inverse
transformation method, the random variable $X=Q\left(  U\right)  \sim
MOEGG\left(  \alpha,\beta,\lambda,\theta\right)  $. We can use this equation
to generate random numbers from the MOEGG distribution when the parameters
$\alpha,\beta,\lambda$ and $\theta$ are known.\bigskip

\subsection{\textit{Mean deviations}}

Let $X\sim MOEGG\left(  \alpha,\beta,\lambda,\theta\right)  $. The mean
deviations of $X$ about the mean $\mu=E(X)$ and about the median $M$ can be
used as measures of spread in a population. They are defined by%
\[
\delta_{1}=\int_{0}^{\infty}\left\vert x-\mu\right\vert f\left(  x\right)
dx\text{ and }\delta_{2}=\int_{0}^{\infty}\left\vert x-M\right\vert f\left(
x\right)  dx,
\]
respectively. These measures can be expressed as%
\[
\delta_{1}=2\mu F\left(  \mu\right)  -2\mu+2\int_{\mu}^{\infty}xf\left(
x\right)  dx\text{ \ and \ }\delta_{2}=-2\mu+\int_{M}^{\infty}xf\left(
x\right)  dx.
\]
We consider only the cas $\theta>1,$ since we can replace $w_{j}$ by $v_{j}$
when $\theta\in\left(  0,1\right)  .$ Then, from (\ref{8}), we have%
\[
\int_{\eta}^{\infty}xf\left(  x\right)  dx=\beta\theta\sum_{i=0}^{\infty
}\alpha\left(  i+1\right)  v_{i}\int_{\eta}^{\infty}xe^{\lambda x}%
e^{-\frac{\beta}{\lambda}\left(  e^{\lambda x}-1\right)  }\left(
1-e^{-\frac{\beta}{\lambda}\left(  e^{\lambda x}-1\right)  }\right)
^{\alpha\left(  i+1\right)  -1}.
\]
By using the binomial series expansion of $\left(  1-e^{-\frac{\beta}{\lambda
}\left(  e^{\lambda x}-1\right)  }\right)  ^{\alpha\left(  i+1\right)  -1}$,
we get%
\[
\int_{\eta}^{\infty}xf\left(  x\right)  dx=\beta\theta\sum_{i,j=0}^{\infty
}\binom{\alpha\left(  i+1\right)  -1}{j}\left(  -1\right)  ^{j}v_{i}%
\alpha\left(  i+1\right)  \int_{\eta}^{\infty}xe^{\lambda x}e^{-\frac{\left(
j+1\right)  \beta}{\lambda}\left(  e^{\lambda x}-1\right)  }dx.
\]
By using the substitution $z=e^{\lambda x}$ in the above integral, we can get%
\begin{align*}
\int_{\eta}^{\infty}xf\left(  x\right)  dx  &  =\frac{\beta\theta}{\lambda
^{2}}\sum_{i,j=0}^{\infty}\binom{\alpha\left(  i+1\right)  -1}{j}\left(
-1\right)  ^{j}\alpha\left(  i+1\right)  v_{i}e^{\frac{\beta}{\lambda}\left(
j+1\right)  }\\
&  \times\int_{e^{\lambda\eta}}^{\infty}e^{-\frac{\beta}{\lambda}\left(
j+1\right)  z}\ln zdx.
\end{align*}
By changing variable and integrating by parts yield%
\begin{align*}
\int_{\eta}^{\infty}xf\left(  x\right)  dx  &  =\frac{\beta\theta}{\lambda
^{2}}\sum_{i,j=0}^{\infty}\binom{\alpha\left(  i+1\right)  -1}{j}\left(
-1\right)  ^{j}\alpha\left(  i+1\right)  v_{i}e^{\frac{\beta}{\lambda}\left(
j+1\right)  }\\
&  \times\left[  \lambda\eta e^{-\frac{\beta}{\lambda}\left(  j+1\right)
e^{\lambda\eta}}+\Gamma\left(  0,\frac{\beta}{\lambda}\left(  j+1\right)
e^{\lambda\eta}\right)  \right]  ,
\end{align*}
where $\Gamma\left(  u,v\right)  =\int_{v}^{\infty}x^{u-1}e^{-x}dx$ is the
incomplete gamma function.\medskip

\noindent Therefore
\begin{align*}
\delta_{1}  &  =2\mu F\left(  \mu\right)  -2\mu+\frac{2\beta\theta}%
{\lambda^{2}}\sum_{i,j=0}^{\infty}\binom{\alpha\left(  i+1\right)  -1}%
{j}\left(  -1\right)  ^{j}\alpha\left(  i+1\right)  v_{i}e^{\frac{\beta
}{\lambda}\left(  j+1\right)  }\\
&  \times\left[  \lambda\mu e^{-\frac{\beta}{\lambda}\left(  j+1\right)
e^{\lambda\mu}}+\Gamma\left(  0,\frac{\beta}{\lambda}\left(  j+1\right)
e^{\lambda\mu}\right)  \right]  ,
\end{align*}
and
\begin{align*}
\delta_{2}  &  =-2\mu+\frac{\beta\theta}{\lambda^{2}}\sum_{i,j=0}^{\infty
}\binom{\alpha\left(  i+1\right)  -1}{j}\left(  -1\right)  ^{j}\alpha\left(
i+1\right)  v_{i}e^{\frac{\beta}{\lambda}\left(  j+1\right)  }\\
&  \times\left[  \lambda Me^{-\frac{\beta}{\lambda}\left(  j+1\right)
e^{\lambda M}}+\Gamma\left(  0,\frac{\beta}{\lambda}\left(  j+1\right)
e^{\lambda M}\right)  \right]  .
\end{align*}

\bigskip

\noindent Table 2 gives $\delta_{1}$ and $\ \delta_{2}$ for selected parameter
values of the $MOEGG\left(  \alpha,\beta,\lambda,\theta\right)  $ distribution.

\begin{center}%
\begin{table}[h] \centering
\caption{Means deviations of the MOEGG distribution.}%
\begin{tabular}
[c]{cccccc}\hline
$\alpha$ & $\beta$ & $\lambda$ & $\theta$ & $\delta_{1}$ & $\delta_{2}%
$\\\hline
$2$ & $1.5$ & $0.5$ & $0.5$ & $0.31524$ & $0.01137$\\
$2$ & $1.5$ & $0.5$ & $0.8$ & $0.34505$ & $0.10172$\\
$2$ & $1.5$ & $0.5$ & $5$ & $0.40255$ & $1.19131$\\
$2$ & $1.5$ & $0.5$ & $15$ & $0.39132$ & $1.50232$\\\hline
\end{tabular}%
\end{table}%
\bigskip
\end{center}

\subsection{\textit{Mean residual life}}

The mean residual life, also known as the mean remaining life, plays an
important role in many fields such as industrial reliability, biomedical
science, life insurance and demography among others. The mean residual life
function$,$ at point $t$, of a lifetime random variable $X\sim MOEGG\left(
\alpha,\beta,\lambda,\theta\right)  $ is%
\begin{align*}
\mu\left(  t\right)   &  =\frac{1}{1-F\left(  t\right)  }\int_{t}^{\infty
}\left[  1-F\left(  x\right)  \right]  dx\\
&  =\frac{\theta+\overline{\theta}\left(  1-e^{-\frac{\beta}{\lambda}\left(
e^{\lambda t}-1\right)  }\right)  ^{\alpha}}{-1+2\theta+\overline{\theta
}\left(  1-e^{-\frac{\beta}{\lambda}\left(  e^{\lambda t}-1\right)  }\right)
^{\alpha}}\int_{t}^{\infty}\frac{-1+2\theta+\overline{\theta}\left(
1-e^{-\frac{\beta}{\lambda}\left(  e^{\lambda x}-1\right)  }\right)  ^{\alpha
}}{\theta+\overline{\theta}\left(  1-e^{-\frac{\beta}{\lambda}\left(
e^{\lambda x}-1\right)  }\right)  ^{\alpha}}dx,\text{ \ }t\geq0.
\end{align*}

\noindent Numerical values of $\mu\left(  5\right)  $ and $\mu\left(
0.5\right)  $ for some values of the parameters $\alpha,\beta,\lambda$ and
$\theta$ are presented in Table 3.

\begin{center}%
\begin{table}[h] \centering
\caption{Mean residual life of the MOEGG distribution.}%
\begin{tabular}
[c]{cccccc}\hline
$\alpha$ & $\beta$ & $\lambda$ & $\theta$ & $\mu\left(  5\right)  $ &
$\mu\left(  0.5\right)  $\\\hline
$2.5$ & $1.5$ & $0.5$ & $0.1$ & $0.05543$ & $0.29961$\\
$0.5$ & $0.25$ & $0.3$ & $0.1$ & $0.71577$ & $1.09276$\\
$2.5$ & $1.5$ & $0.5$ & $0.8$ & $0.05029$ & $0.47139$\\
$0.5$ & $0.25$ & $0.3$ & $0.8$ & $0.72227$ & $1.57295$\\
$2.5$ & $1.5$ & $0.5$ & $2.5$ & $0.04995$ & $0.67338$\\
$0.5$ & $0.25$ & $0.3$ & $2.5$ & $0.73767$ & $2.21551$\\
$2.5$ & $1.5$ & $0.5$ & $8.5$ & $0.04984$ & $0.95539$\\
$0.5$ & $0.25$ & $0.3$ & $8.5$ & $0.78810$ & $3.16767$\\\hline
\end{tabular}%
\end{table}%

\end{center}

\subsection{\textit{R\'{e}nyi entropy}}

The entropy of a random variable is a measure of uncertainty variation and has
been used in various situations in science and engineering. In the literature,
several measures of entropy have been studied. Here, we give the R\'{e}nyi
entropy for the MOEGG\ distribution. The R\'{e}nyi entropy for the MOEGG
distribution is%
\[
I_{R}\left(  s\right)  =\frac{1}{1-s}\ln\left(  \int_{0}^{\infty}f^{s}\left(
x\right)  dx\right)  ,\text{ }s>0,\text{ }s\neq1,
\]
where%
\[
f^{s}\left(  x\right)  =\frac{\left(  \alpha\beta\theta\right)  ^{s}e^{\lambda
sx}e^{-\frac{s\beta}{\lambda}\left(  e^{\lambda x}-1\right)  }\left(
1-e^{-\frac{\beta}{\lambda}\left(  e^{\lambda x}-1\right)  }\right)
^{s\left(  \alpha-1\right)  }}{\left[  \theta+\overline{\theta}\left(
1-e^{-\frac{\beta}{\lambda}\left(  e^{\lambda x}-1\right)  }\right)  ^{\alpha
}\right]  ^{2s}}.
\]
By applying the expansion (\ref{6}) to the denominatorin the previous
equation, we can rewrite after some algebra $f^{s}\left(  x\right)  $, for
$\theta\in\left(  0,1\right)  ,$ as%
\[
f^{s}\left(  x\right)  =\left(  \alpha\beta\theta\right)  ^{s}\sum
_{k=0}^{\infty}\sum_{i=0}^{\infty}\sum_{j=i}^{\infty}\binom{2s+j-1}{j}%
\binom{j}{i}\binom{\alpha i}{k}\overline{\theta}^{j}\left(  -1\right)
^{k+i}e^{\lambda sx}e^{-\frac{\beta}{\lambda}\left(  k+s\right)  \left(
e^{\lambda x}-1\right)  },
\]
and for $\theta>1,$ we get%
\begin{align*}
f^{s}\left(  x\right)   &  =\left(  \frac{\alpha\beta}{\theta}\right)
^{s}\sum_{k=0}^{\infty}\sum_{i=0}^{\infty}\sum_{j=i}^{\infty}\binom{2s+j-1}%
{j}\binom{\alpha\left(  j+s\right)  -s}{k}\\
&  \times\left(  1-\theta^{-1}\right)  ^{j}\left(  -1\right)  ^{k}e^{\lambda
sx}e^{-\frac{\beta}{\lambda}\left(  k+s\right)  \left(  e^{\lambda
x}-1\right)  }.
\end{align*}
\bigskip Thus, the R\'{e}nyi entropy is%
\[
I_{R}\left(  s\right)  =\frac{1}{1-s}\ln\left(  \int_{0}^{\infty}f^{s}\left(
x\right)  dx\right)  ,
\]
where, if $\theta\in\left(  0,1\right)  $%
\begin{align*}
\int_{0}^{\infty}f^{\gamma}\left(  x\right)  dx  &  =\frac{\left(
\alpha\lambda\theta\right)  ^{s}}{\lambda\left(  k+s\right)  ^{s}}\sum
_{k=0}^{\infty}\sum_{i=0}^{\infty}\sum_{j=i}^{\infty}\binom{2s+j-1}{j}%
\binom{j}{i}\binom{\alpha i}{k}\\
&  \times\overline{\theta}^{j}\left(  -1\right)  ^{k+i}e^{\frac{\beta}%
{\lambda}\left(  k+s\right)  }\Gamma\left(  s,\frac{\beta}{\lambda}\left(
k+s\right)  \right)  ,
\end{align*}
and if $\theta>1$%
\begin{align*}
\int_{0}^{\infty}f^{\gamma}\left(  x\right)  dx  &  =\frac{\lambda^{s-1}%
\alpha^{s}}{\left(  \theta k+\theta s\right)  ^{s}}\sum_{k=0}^{\infty}%
\sum_{i=0}^{\infty}\sum_{j=i}^{\infty}\binom{2s+j-1}{j}\binom{\alpha\left(
j+s\right)  -s}{k}\\
&  \times\left(  1-\theta^{-1}\right)  ^{j}\left(  -1\right)  ^{k}%
e^{\frac{\beta}{\lambda}\left(  k+s\right)  }\Gamma\left(  s,\frac{\beta
}{\lambda}\left(  k+s\right)  \right)  ,
\end{align*}
with $\Gamma\left(  u,v\right)  =\int_{v}^{\infty}x^{u-1}e^{-x}dx\ $is the
incomplete gamma function.\bigskip\bigskip

\noindent Table 4 gives $I_{R}\left(  0.2\right)  $ of the MOEGG distribution
for different choices of parameters $\alpha,\beta,\lambda$\ and $\theta$.

\begin{center}%
\begin{table}[h] \centering
\caption{Rényi entropy of the MOEGG distribution.}%
\begin{tabular}
[c]{ccccc}\hline
$\alpha$ & $\beta$ & $\lambda$ & $\theta$ & $I_{R}\left(  0.2\right)
$\\\hline
$5$ & $2$ & $1.5$ & $0.1$ & $0.05879$\\
$0.2$ & $0.7$ & $0.9$ & $0.1$ & $0.12745$\\
$5$ & $2$ & $1.5$ & $0.8$ & $0.14771$\\
$0.2$ & $0.7$ & $0.9$ & $0.8$ & $0.42179$\\
$5$ & $2$ & $1.5$ & $1.2$ & $0.15798$\\
$0.2$ & $0.7$ & $0.9$ & $1.2$ & $0.47002$\\
$5$ & $2$ & $1.5$ & $7.5$ & $0.17568$\\
$0.2$ & $0.7$ & $0.9$ & $7.5$ & $0.62701$\\\hline
\end{tabular}%
\end{table}%

\end{center}

\subsection{\textit{Order statistics}}

The order statistics play an important role in reliability and life testing.
Let $X_{1},\ldots,X_{n}$ be a simple random sample from MOEGG distribution
with cdf and pdf as in (\ref{3}) and (\ref{4}), respectively. Let $X_{1,n}%
\leq\cdots\leq X_{n,n}$ denote the order statistics obtained from this sample.
In reliability literature, the $k$th order statistic, say $X_{k;n}$, denotes
the lifetime of an $(n-k+1)$--out--of--$n$ system which consists of $n$
independent and identically components. The pdf of $X_{k;n}$ is given by%
\[
f_{k,n}\left(  x\right)  =\frac{n!}{\left(  n-k\right)  !\left(  k-1\right)
!}f\left(  x\right)  \left[  F\left(  x\right)  \right]  ^{k-1}\left[
1-F\left(  x\right)  \right]  ^{n-k},\text{ for }k=1,\ldots,n.
\]
Since $0<F(x)<1$ for $x>0$, then by using the binomial series expansion of
$\left[  1-F\left(  x\right)  \right]  ^{n-k}$, we obtain%
\[
f_{k,n}\left(  x\right)  =\frac{n!}{\left(  n-k\right)  !\left(  k-1\right)
!}\sum_{j=0}^{n-k}\binom{n-k}{j}\left(  -1\right)  ^{j}f\left(  x\right)
\left[  F\left(  x\right)  \right]  ^{^{k+j-1}}.
\]
Then from (\ref{2}), we have%
\[
f_{k,n}\left(  x\right)  =\frac{\theta n!g\left(  x\right)  }{\left(
k-1\right)  !\left(  n-k\right)  !}\sum_{j=0}^{n-k}\binom{n-k}{j}\left(
-1\right)  ^{j}\frac{\left[  G\left(  x\right)  \right]  ^{k+j-1}}{\left[
\theta+\overline{\theta}G\left(  x\right)  \right]  ^{k+j+1}}.
\]
For $\theta\in(0,1)$, by using expansion (\ref{6}) in the above equation and
after some algebra, we get%
\begin{equation}
f_{k,n}\left(  x\right)  =\sum_{i=0}^{\infty}\sum_{j=0}^{n-k}\sum_{l=0}%
^{i}\sum_{m=0}^{i+j-l+k-1}c_{ijlm}g\left(  x;\alpha\left(  m+1\right)
,\beta,\lambda\right)  , \label{9}%
\end{equation}
where%
\[
c_{ijlm}=\frac{n!\theta\left(  1-\theta\right)  ^{j}\binom{i}{l}\binom
{i+j+k}{i}\binom{i+j-l+k-1}{m}\left(  -1\right)  ^{i+j-l+m}}{\left(
k-1\right)  !\left(  n-k\right)  !\left(  m+1\right)  !}.
\]
In a similar way, if $\theta>1,$ we get%
\begin{equation}
f_{k,n}\left(  x\right)  =\sum_{i=0}^{\infty}\sum_{j=0}^{n-k}\sum
_{l=0}^{i+j+k-1}t_{ijl}g\left(  x;\alpha\left(  l+1\right)  ,\beta
,\lambda\right)  , \label{10}%
\end{equation}
where%
\[
t_{ijl}=\frac{n!\left(  \theta-1\right)  ^{j}\binom{i+j+k}{i}\binom
{i+j+k-1}{m}\left(  -1\right)  ^{j+l}}{\theta^{i+j+k}\left(  k-1\right)
!\left(  n-k\right)  !\left(  l+1\right)  !}.
\]
It is clear that%
\[
\sum_{i=0}^{\infty}\sum_{j=0}^{n-k}\sum_{l=0}^{i}\sum_{m=0}^{i+j-l+k-1}%
c_{ijlm}=\sum_{i=0}^{\infty}\sum_{j=0}^{n-k}\sum_{l=0}^{i+j+k-1}t_{ijl}=1.
\]

\noindent Equations (\ref{9}) and (\ref{10}) show that the pdf of the MOEGG
order statistics is an infinite linear combination of GG density functions.
So, several mathematical properties of the these order statistics can be
obtained directly from the properties of the GG such as the ordinary moments,
inverse and factorial moments, moment generating function, mean deviations.
The $r$th moment of the $k$th order statistic is determined from (\ref{9}) and
(\ref{10}) as%
\[
E\left(  X_{k,n}^{r}\right)  =\sum_{i=0}^{\infty}\sum_{j=0}^{n-k}\sum
_{l=0}^{i}\sum_{m=0}^{i+j-l+k-1}c_{ijlm}E\left(  Y_{m+1}^{r}\right)  ,\text{
\ if \ }\theta\in\left(  0,1\right)  ,
\]
and%
\[
E\left(  X_{k,n}^{r}\right)  =\sum_{i=0}^{\infty}\sum_{j=0}^{n-k}\sum
_{l=0}^{i+j+k-1}t_{ijl}E\left(  Y_{l+1}^{r}\right)  ,\text{ \ if \ }\theta>1,
\]
where%
\[
E\left(  X_{i+1}^{r}\right)  =\frac{\alpha\left(  i+1\right)  r!}{\lambda^{r}%
}\sum_{j=0}^{\infty}\binom{\alpha\left(  i+1\right)  -1}{j}\frac{\left(
-1\right)  ^{j}e^{\frac{\beta}{\lambda}\left(  j+1\right)  }}{\left(
j+1\right)  }E_{1}^{r-1}\left(  \frac{\beta}{\lambda}\left(  j+1\right)
\right)  .
\]
\bigskip

\subsection{\textit{Stochastic orderings}}

Stochastic order is the most common notion of comparison of random variables.
This concept can be applied in several area, such as insurance, operations
research, queuing theory, survival analysis and reliability theory (see
\cite{r21}). Let $X$ and $Y$ be two random variables with distribution
functions $F$ and $G$ and density functions $f$ and $g$, respectively. We say that

\begin{enumerate}
\item $X$ is smaller than $Y$ in the stochastic\ order, denoted by $X\leq
_{st}Y$, if $F(x)\leq G(x)$ for all $x$.

\item $X$ is smaller than $Y$ in the likelihood ratio order, denoted by
$X\leq_{lr}Y$, if $f(x)/g(x)$ is decreasing in $x\geq0$,

\item $X$ is smaller than $Y$ in the hazard rate order, denoted by $X\leq
_{hr}Y$, if $\left(  1-F(x)\right)  /\left(  1-G(x)\right)  $ is decreasing in
$x\geq0$.

\item $X$ is smaller than $Y$ in the reversed hazard rate order, denoted by
$X\leq_{rhr}Y$, if $F(x)/G(x)$ is decreasing in $x\geq0$.
\end{enumerate}

We have the following implications (see \cite{r21}),%
\begin{equation}
X\leq_{rhr}Y\Longleftarrow X\leq_{lr}Y\Longrightarrow X\leq_{hr}%
Y\Longrightarrow X\leq_{st}Y. \label{11}%
\end{equation}

\begin{theorem}
Let $X\sim MOEGG\left(  \alpha,\beta,\lambda,\theta_{1}\right)  $ and let
$Y\sim MOEGG\left(  \alpha,\beta,\lambda,\theta_{2}\right)  $. If $\theta
_{1}<\theta_{2}$, then%
\[
X\leq_{lr}Y\text{, }X\leq_{rhr}Y,\text{ }X\leq_{hr}Y\text{\ \ and \ }%
X\leq_{st}Y.
\]

\end{theorem}

\begin{proof}
We have%
\[
\frac{f\left(  x\right)  }{g\left(  x\right)  }=\frac{\theta_{1}}{\theta_{2}%
}\left[  \frac{1-\overline{\theta}_{2}\left(  1-\left(  1-t\right)  ^{\alpha
}\right)  }{1-\overline{\theta}_{1}\left(  1-\left(  1-t\right)  ^{\alpha
}\right)  }\right]  ^{2},
\]
where $t=e^{-\frac{\beta}{\lambda}\left(  e^{\lambda x}-1\right)  }$. It easy
to verify%
\[
\frac{d}{dx}\left[  \frac{f\left(  x\right)  }{g\left(  x\right)  }\right]
=2\frac{\theta_{1}}{\theta_{2}}\left(  \theta_{1}-\theta_{2}\right)
\frac{\left(  \theta_{2}+\overline{\theta}_{2}\left(  1-t\right)  ^{\alpha
}\right)  \alpha\beta e^{\lambda x}t\left(  1-t\right)  ^{\alpha-1}}{\left(
\theta_{1}+\overline{\theta}_{1}\left(  1-t\right)  ^{\alpha}\right)  ^{3}}.
\]
Then, if $\theta_{1}<\theta_{2},$ we have $\frac{d}{dx}\left[  \frac{f\left(
x\right)  }{g\left(  x\right)  }\right]  <0$. So, $\frac{f\left(  x\right)
}{g\left(  x\right)  }$ is decreasing in $x$ i.e., $X\leq_{lr}Y$. From the
implications (\ref{11}), we get the remaining statements.\bigskip
\end{proof}

\section{\textbf{Estimation\label{sec 4}}}

In this section we present the maximum likelihood estimate (MLE) and derive
the asymptotic confidence intervals of the unknown parameter vector
$\mathbf{\Theta}=(\alpha,\beta,\lambda,\theta)^{T}$ of the MOEGG distribution.
Let $x_{1},\ldots,x_{n}$ be a random sample of size $n$ from MOEGG
distribution with pdf (\ref{4}), then the likelihood function is%
\begin{equation}
L\left(  \mathbf{\Theta}\right)  =%
{\displaystyle\prod\limits_{i=1}^{n}}
\frac{\alpha\beta\theta e^{\lambda x_{i}}e^{-\frac{\beta}{\lambda}\left(
e^{\lambda x_{i}}-1\right)  }\left(  1-e^{-\frac{\beta}{\lambda}\left(
e^{\lambda x_{i}}-1\right)  }\right)  ^{\alpha-1}}{\left[  \theta
+\overline{\theta}\left(  1-e^{-\frac{\beta}{\lambda}\left(  e^{\lambda x_{i}%
}-1\right)  }\right)  ^{\alpha}\right]  ^{2}}. \label{12}%
\end{equation}
Taking the logarithm of equation(\ref{12}), we get the log-likelihood function%
\begin{align*}
\ell\left(  \mathbf{\Theta}\right)   &  =n\ln\alpha+n\ln\beta+n\ln
\theta+\lambda\sum_{i=1}^{n}x_{i}-\frac{\beta}{\lambda}\sum_{i=1}^{n}\left(
e^{\lambda x_{i}}-1\right) \\
&  +\left(  \alpha-1\right)  \sum_{i=1}^{n}\ln\left(  1-e^{-\frac{\beta
}{\lambda}\left(  e^{\lambda x_{i}}-1\right)  }\right) \\
&  -2\sum_{i=1}^{n}\ln\left(  \theta+\overline{\theta}\left(  1-e^{-\frac
{\beta}{\lambda}\left(  e^{\lambda x_{i}}-1\right)  }\right)  ^{\alpha
}\right)  .
\end{align*}
The components of the score vector are given, by taking the partial
derivatives of $\ell\left(  \mathbf{\Theta}\right)  $ with respect to
$\alpha,$ $\beta,$ $\lambda$ and $\theta,$ as follows%
\begin{align*}
U_{\alpha}  &  =\frac{n}{\alpha}+\sum_{i=1}^{n}\ln\left(  1-e^{-\frac{\beta
}{\lambda}\left(  e^{\lambda x_{i}}-1\right)  }\right) \\[0.15cm]
&  -2\sum_{i=1}^{n}\frac{\overline{\theta}\left(  1-e^{-\frac{\beta}{\lambda
}\left(  e^{\lambda x_{i}}-1\right)  }\right)  ^{\alpha}\ln\left(
1-e^{-\frac{\beta}{\lambda}\left(  e^{\lambda x_{i}}-1\right)  }\right)
}{\theta+\overline{\theta}\left(  1-e^{-\frac{\beta}{\lambda}\left(
e^{\lambda x_{i}}-1\right)  }\right)  ^{\alpha}},
\end{align*}%
\begin{align*}
U_{\beta}  &  =\frac{n}{\beta}-\frac{1}{\lambda}\sum_{i=1}^{n}\left(
e^{\lambda x_{i}}-1\right)  +\left(  \alpha-1\right)  \sum_{i=1}^{n}%
\frac{\left(  e^{\lambda x_{i}}-1\right)  e^{-\frac{\beta}{\lambda}\left(
e^{\lambda x_{i}}-1\right)  }}{\lambda\left(  1-e^{-\frac{\beta}{\lambda
}\left(  e^{\lambda x_{i}}-1\right)  }\right)  }\\
&  -2\sum_{i=1}^{n}\frac{\overline{\theta}\alpha\left(  e^{\lambda x_{i}%
}-1\right)  e^{-\frac{\beta}{\lambda}\left(  e^{\lambda x_{i}}-1\right)
}\left(  1-e^{-\frac{\beta}{\lambda}\left(  e^{\lambda x_{i}}-1\right)
}\right)  ^{\alpha-1}}{\lambda\left(  \theta+\overline{\theta}\left(
1-e^{-\frac{\beta}{\lambda}\left(  e^{\lambda x_{i}}-1\right)  }\right)
^{\alpha}\right)  },
\end{align*}%
\begin{align*}
U_{\lambda}  &  =\sum_{i=1}^{n}x_{i}+\frac{\beta}{\lambda^{2}}\sum_{i=1}%
^{n}\left(  e^{\lambda x_{i}}-1\right)  -\frac{\beta}{\lambda}\sum_{i=1}%
^{n}x_{i}e^{\lambda x_{i}}\\
&  +\left(  \alpha-1\right)  \beta\sum_{i=1}^{n}\frac{\left(  \lambda
x_{i}e^{\lambda x_{i}}-e^{\lambda x_{i}}+1\right)  e^{-\frac{\beta}{\lambda
}\left(  e^{\lambda x_{i}}-1\right)  }}{\lambda^{2}\left(  1-e^{-\frac{\beta
}{\lambda}\left(  e^{\lambda x_{i}}-1\right)  }\right)  },\\
&  -2\sum_{i=1}^{n}\frac{\alpha\beta\overline{\theta}\left(  \lambda
x_{i}e^{\lambda x_{i}}-e^{\lambda x_{i}}+1\right)  e^{-\frac{\beta}{\lambda
}\left(  e^{\lambda x_{i}}-1\right)  }\left(  1-e^{-\frac{\beta}{\lambda
}\left(  e^{\lambda x_{i}}-1\right)  }\right)  ^{\alpha-1}}{\lambda^{2}\left(
\theta+\overline{\theta}\left(  1-e^{-\frac{\beta}{\lambda}\left(  e^{\lambda
x_{i}}-1\right)  }\right)  ^{\alpha}\right)  },
\end{align*}
and%
\[
U_{\theta}=\frac{n}{\theta}-2\sum_{i=1}^{n}\frac{1-\left(  1-e^{-\frac{\beta
}{\lambda}\left(  e^{\lambda x_{i}}-1\right)  }\right)  ^{\alpha}}%
{\theta+\overline{\theta}\left(  1-e^{-\frac{\beta}{\lambda}\left(  e^{\lambda
x_{i}}-1\right)  }\right)  ^{\alpha}}.
\]
\bigskip

\noindent The MLE $\widehat{\mathbf{\Theta}}=(\widehat{\alpha},\widehat{\beta
},\widehat{\lambda},\widehat{\theta})^{T}$ of $\mathbf{\Theta}=(\alpha
,\beta,\lambda,\theta)^{T}$ is obtained by solving the nonlinear likelihood
equations $U_{\alpha}=0,U_{\beta}=0,U_{\lambda}=0$ and $U_{\theta}=0.$ Since
these equations cannot be solved analytically, then we use the statistical
software to solve them numerically via iterative methods such as a
Newton--Raphson technique.\medskip

\noindent The normal approximation for the MLE of $\mathbf{\Theta}$ is used
for constructing approximate confidence intervals, confidence regions and
testing hypotheses of the parameters $\alpha,\beta,\lambda$ and $\theta.$
Under conditions that are fulfilled for parameters in the interior of the
parameter space but not on the boundary, the asymptotic distribution of
$\sqrt{n}\left(  \widehat{\mathbf{\Theta}}-\mathbf{\Theta}\right)  \ $is
$\mathcal{N}_{4}\left(  \mathbf{0},I^{-1}\left(  \mathbf{\Theta}\right)
\right)  ,$ where $I\left(  \mathbf{\Theta}\right)  $ is the expected
information matrix (see \cite[Chapter 9]{r5}). This asymptotic behavior is
valid if $I\left(  \mathbf{\Theta}\right)  $ is replaced by $J\left(
\widehat{\mathbf{\Theta}}\right)  $, where $J\left(  \mathbf{\Theta}\right)  $
is the observed information matrix, given by%
\[
J\left(  \mathbf{\Theta}\right)  =-\left(
\begin{array}
[c]{cccc}%
U_{\alpha\alpha} & U_{\alpha\beta} & U_{\alpha\lambda} & U_{\alpha\theta}\\
. & U_{\beta\beta} & U_{\beta\lambda} & U_{\beta\theta}\\
. & . & U_{\lambda\lambda} & U_{\lambda\theta}\\
. & . & . & U_{\theta\theta}%
\end{array}
\right)  ,
\]
whose elements are given in the Appendix.\bigskip

\noindent Therefore, an $100\left(  1-\omega\right)  \%$ asymptotic confidence
intervals for the parameters $\alpha,\ \beta,$ $\lambda\ $and $\theta$ are
given by%
\[
\widehat{\alpha}\pm Z_{\frac{\omega}{2}}\sqrt{var\left(  \widehat{\alpha
}\right)  },\text{\ }\widehat{\beta}\pm Z_{\frac{\omega}{2}}\sqrt{var\left(
\widehat{\beta}\right)  },\text{\ }\widehat{\lambda}\pm Z_{\frac{\omega}{2}%
}\sqrt{var\left(  \widehat{\lambda}\right)  },\ \text{and\ \ }\widehat{\theta
}\pm Z_{\frac{\omega}{2}}\sqrt{var\left(  \widehat{\theta}\right)  },
\]
respectively, where $var\left(  \cdot\right)  $ is the diagonal element of
$J^{-1}\left(  \widehat{\mathbf{\Theta}}\right)  $ corresponding to each
parameter and $Z_{\frac{\omega}{2}}$ is the quantile $(1-\omega/2)\ $of the
standard normal distribution.\bigskip

\section{\textbf{Application\label{sec 5}}}

In this section, we give an application of the MOEGG distribution to one real
data set to demonstrate its superiority. The data, given in Table 5, represent
the strengths of 1.5 cm glass fibres measured at the National Physical
Laboratory, England (see \cite{r22}). Table 6 gives a descriptive summary of
the data.

\begin{center}%
\begin{table}[h] \centering
\caption{The strengths of glass fibers.}%
\begin{tabular}
[c]{lllllllllll}\hline
0.55 & 0.93 & 1.25 & 1.36 & 1.49 & 1.52 & 1.58 & 1.61 & 1.64 & 1.68 & 1.73\\
1.04 & 1.27 & 1.39 & 1.49 & 1.53 & 1.59 & 1.61 & 1.66 & 1.68 & 1.76 & 1.82\\
1.28 & 1.42 & 1.50 & 1.54 & 1.60 & 1.62 & 1.66 & 1.69 & 1.76 & 1.84 & 2.24\\
1.48 & 1.50 & 1.55 & 1.61 & 1.62 & 1.66 & 1.70 & 1.77 & 1.84 & 0.84 & 1.24\\
1.55 & 1.61 & 1.63 & 1.67 & 1.70 & 1.78 & 1.89 & 1.81 & 2.00 & 0.74 & 2.01\\
0.77 & 1.11 & 0.81 & 1.13 & 1.29 & 1.30 & 1.48 & 1.51 &  &  & \\\hline
\end{tabular}%
\end{table}%
\bigskip%

\begin{table}[h] \centering
\caption{Descriptive statistics.}%
\begin{tabular}
[c]{lllllll}\hline
Mean & Variance & Median & Kurtosis & Skewness & Min. & Max.\\\hline
\multicolumn{1}{c}{$1.51$} & \multicolumn{1}{c}{$0.11$} &
\multicolumn{1}{c}{$1.59$} & \multicolumn{1}{c}{$3.92$} &
\multicolumn{1}{c}{$-0.90$} & \multicolumn{1}{c}{$0.55$} &
\multicolumn{1}{c}{$2.24$}\\\hline
\end{tabular}%
\end{table}%
\bigskip
\end{center}

\noindent We compare the MOEGG model with the Gompertz, GG, BG, McG,
Marshall-Olkin extend generalized exponential (MOEGE), generalized exponential
(GE) and Marshall-Olkin generalized Lindley (MOEGL) models. The pdf of the
GG\ (see \cite{r6}) distribution is%
\[
f\left(  x\right)  =\alpha\beta e^{\lambda x}e^{-\frac{\beta}{\lambda}\left(
e^{\lambda x}-1\right)  }\left(  1-e^{-\frac{\beta}{\lambda}\left(  e^{\lambda
x}-1\right)  }\right)  ^{\alpha-1},\text{ }x,\alpha,\beta,\lambda>0.
\]
For $\alpha=1$ the GG distribution becomes the Gompertz distribution. The pdf
of the BG distribution (see \cite{r12}) is%
\[
f\left(  x\right)  =\frac{\beta e^{\lambda x}e^{-\frac{\gamma\beta}{\lambda
}\left(  e^{\lambda x}-1\right)  }}{B\left(  \theta,\gamma\right)  }\left(
1-e^{-\frac{\beta}{\lambda}\left(  e^{\lambda x}-1\right)  }\right)
^{\theta-1},\text{ }x,\beta,\lambda,\theta,\gamma>0.
\]
The pdf of the McG distribution (see \cite{r18}) is%
\[
f\left(  x\right)  =\frac{\delta\beta e^{\lambda x}e^{-\frac{\beta}{\lambda
}\left(  e^{\lambda x}-1\right)  }}{B\left(  \theta/\delta,\gamma\right)
}\left(  1-e^{-\frac{\beta}{\lambda}\left(  e^{\lambda x}-1\right)  }\right)
^{\theta-1}\left(  1-\left[  1-e^{-\frac{\beta}{\lambda}\left(  e^{\lambda
x}-1\right)  }\right]  ^{\delta}\right)  ^{\gamma-1},
\]
for $x,\beta,\lambda,\theta,\gamma,\delta>0.$ The pdf of the MOEGL
distribution (see \cite{r3}) is%
\[
f\left(  x\right)  =\frac{\alpha\theta\lambda^{2}e^{-\lambda x}\left(
1+x\right)  \left(  1-\frac{1+\lambda+\lambda x}{1+\lambda}e^{-\lambda
x}\right)  ^{\alpha-1}}{\left(  1+\lambda\right)  \left[  \theta
+\overline{\theta}\left(  1-\frac{1+\lambda+\lambda x}{1+\lambda}e^{-\lambda
x}\right)  ^{\alpha}\right]  ^{2}},\text{ }x,\alpha,\lambda,\theta>0.
\]
The pdf of the MOEGE distribution (see \cite{r20}) is%
\[
f\left(  x\right)  =\frac{\alpha\lambda\theta e^{-\lambda x}\left(
1+x\right)  \left(  1-e^{-\lambda x}\right)  ^{a-1}}{\left(  \theta
+\overline{\theta}\left(  1-e^{-\theta x}\right)  ^{\alpha}\right)  ^{2}%
},\text{ }x,\alpha,\lambda,\theta>0.
\]
For $\theta=1$ the MOEGE distribution becomes the GE distribution.\medskip

\noindent The model selection is carried out using the maximized loglikelihood
($-\ln(L)$), Akaike Information Criterion (AIC), Consistent Akaike Information
Criteria (CAIC), Bayesian Information Criteria (BIC), Kolmogorov-Smirnov (K-S)
statistic with its respective p-value, Cram\'{e}r-von Mises (CM) statistic and
Anderson-Darling (AD) statistic. The better distribution to fit the data
corresponds to smaller values of these statisticsThe large p-value for the K-S
test. The MLEs and the corresponding standard errors (SEs) in parenthese are
listed in Table 7. The values of the $-$log$(L)$, AIC, BIC, CAIC and K-S with
its respective p-value are listed in Table 8 whereas the values of the CM and
AD are given in Table 9\textbf{. }So\textbf{, }we conclude that the MOEGG
distribution gives an excellent fit for the data set. In addition, the plots
of the densities together with the data histogram (Figure 3), cdfs with
empirical distribution function (Figure 4) and the probability plots (Figure
5) confirm that the MOEGG model yields a better fit.

\begin{center}%
\begin{table}[h] \centering
\caption{MLEs of the model parameters and the corresponding standard errors
given in parentheses.}%
\begin{tabular}
[c]{lcccccc}\hline
Model & $\alpha$ & $\beta$ & $\lambda$ & $\theta$ & $\gamma$ & $\delta
$\\\hline
MOEGG & $2.2193$ & $0.6791$ & $1.3929$ & $19.1052$ & $-$ & $-$\\
& $\left(  1.9901\right)  $ & $\left(  0.7286\right)  $ & $\left(
0.6427\right)  $ & $\left(  28.2384\right)  $ & $-$ & $-$\\
MOEGE & $34.8741$ & $-$ & $5.6105$ & $158.5649$ & $-$ & $-$\\
& $\left(  34.3649\right)  $ & $-$ & $\left(  0.5917\right)  $ & $\left(
160.3973\right)  $ & $-$ & $-$\\
MOEGL & $27.3356$ & $-$ & $6.0619$ & $176.9702$ & $-$ & $-$\\
& $\left(  32.0782\right)  $ & $-$ & $\left(  0.6202\right)  $ & $\left(
215.4523\right)  $ & $-$ & $-$\\
BG & $-$ & $0.0357$ & $2.8487$ & $1.6357$ & $1.0650$ & $-$\\
& $-$ & $\left(  0.0553\right)  $ & $\left(  0.8196\right)  $ & $\left(
0.7610\right)  $ & $\left(  2.0421\right)  $ & $-$\\
McG & $-$ & $0.0699$ & $2.1279$ & $2.0612$ & $2.6692$ & $4.8180$\\
& $-$ & $\left(  0.1200\right)  $ & $\left(  1.3567\right)  $ & $\left(
1.3565\right)  $ & $\left(  4.8985\right)  $ & $\left(  6.5598\right)  $\\
GG & $1.6253$ & $0.0368$ & $2.8647$ & $-$ & $-$ & $-$\\
& $\left(  0.6971\right)  $ & $\left(  0.0432\right)  $ & $\left(
0.6824\right)  $ & $-$ & $-$ & $-$\\
GE & $31.3068$ & $-$ & $2.6106$ & $-$ & $-$ & $-$\\
& $\left(  9.5016\right)  $ & $-$ & $\left(  0.2379\right)  $ & $-$ & $-$ &
$-$\\
Gompertz & $-$ & $0.0091$ & $3.6262$ & $-$ & $-$ & $-$\\
& $-$ & $\left(  0.0049\right)  $ & $\left(  0.3467\right)  $ & $-$ & $-$ &
$-$\\\hline
\end{tabular}%
\end{table}%
%

\begin{table}[h] \centering
\caption{The statistics: -ln(L), AIC, CAIC, BIC, K-S and
p-value.}%
\begin{tabular}
[c]{lcccccc}\hline
Model & $-\ln(L)$ & AIC & CAIC & BIC & K-S & p-value\\\hline
MOEGG & $\mathbf{12.0571}$ & $\mathbf{32.1143}$ & $\mathbf{32.6660}$ &
$\mathbf{40.6868}$ & $\mathbf{0.1008}$ & $\mathbf{0.5443}$\\
MOEGE & $16.1898$ & $38.3796$ & $38.6847$ & $44.8090$ & $0.1353$ & $0.1993$\\
MOEGL & $15.9565$ & $37.9129$ & $38.2180$ & $44.3423$ & $0.1309$ & $0.2304$\\
BG & $14.1443$ & $36.2886$ & $36.8403$ & $44.8611$ & $0.1323$ & $0.2205$\\
McG & $13.8548$ & $37.7097$ & $38.5868$ & $48.4253$ & $0.1316$ & $0.2254$\\
GG & $14.1456$ & $34.2912$ & $34.5963$ & $40.7206$ & $0.1321$ & $0.2213$\\
GE & $31.3835$ & $66.7670$ & $66.9003$ & $71.0532$ & $0.2290$ & $0.0027$\\
Gompertz & $14.8100$ & $33.6200$ & $33.7533$ & $37.9063$ & $0.2382$ & $0.2458
$\\\hline
\end{tabular}
\label{tab4}%
\end{table}%
%

\begin{table}[h] \centering
\caption{The statistics CM and AD.}%
\begin{tabular}
[c]{lcc}\hline
Model & CM & AD\\\hline
MOEGG & $\mathbf{0.0937}$ & $\mathbf{0.5333}$\\
MOEGE & $0.2861$ & $1.5587$\\
MOEGL & $0.2747$ & $1.4968$\\
BG & $0.1624$ & $0.9091$\\
McG & $0.1546$ & $0.8658$\\
GG & $0.1623$ & $0.9085$\\
GE & $0.7798$ & $4.2334$\\
Gompertz & $0.1460$ & $0.8348$\\\hline
\end{tabular}%
\end{table}%

\end{center}

\noindent The asymptotic variance-covariance matrix of the MLEs of the MOEGG
model parameters, which is the inverse of the observed $J^{-1}\left(
\widehat{\mathbf{\Theta}}\right)  $, is%
\[
\left(
\begin{array}
[c]{cccc}%
4.03019 & 0.65668 & -0.59059 & -12.78541\\
0.65668 & 0.51935 & -0.44066 & 16.05558\\
-0.59059 & -0.44066 & 0.38014 & -12.76136\\
-12.78541 & 16.05558 & -12.76136 & 902.45416
\end{array}
\right)  .
\]

\noindent Therefore, the 95\% asymptotic confidence intervals for
$\alpha,\ \beta,$ $\lambda\ $and $\theta$ are $2.2137\pm3.93476,$
$0.7011\pm1.41249,$ $1.3739\pm1.20844\ $and $20.1122\pm58.88011$ respectively.%

{\parbox[b]{4.7202in}{\begin{center}
\includegraphics[
trim=0.000000in 0.209306in 0.357482in 1.110532in,
height=2.8106in,
width=4.7202in
]%
{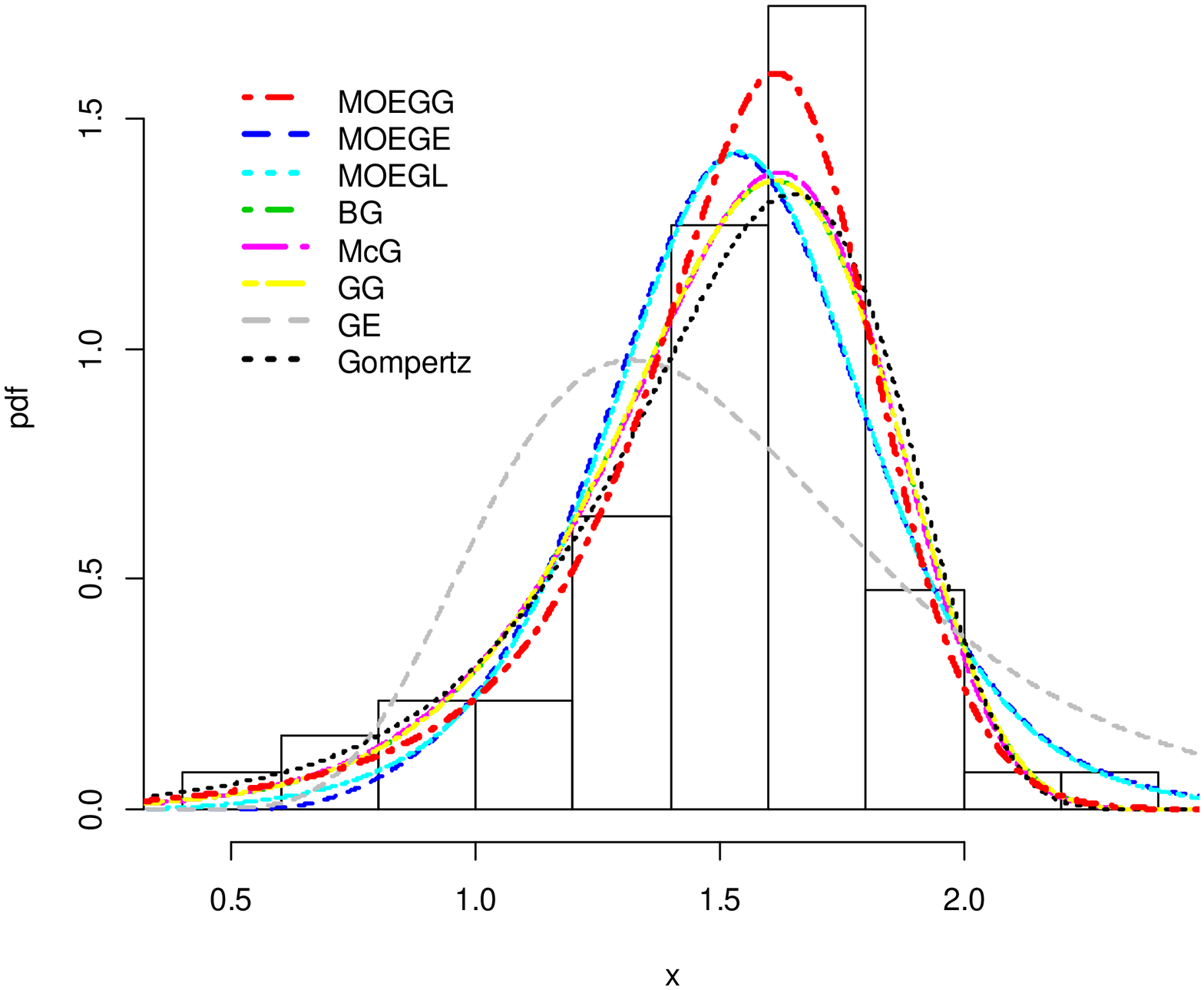}%
\\
Figure 3. Histogram and estimated densities.
\end{center}}}

%

{\parbox[b]{4.7297in}{\begin{center}
\includegraphics[
trim=0.000000in 0.157412in 0.263563in 0.958309in,
height=2.8902in,
width=4.7297in
]%
{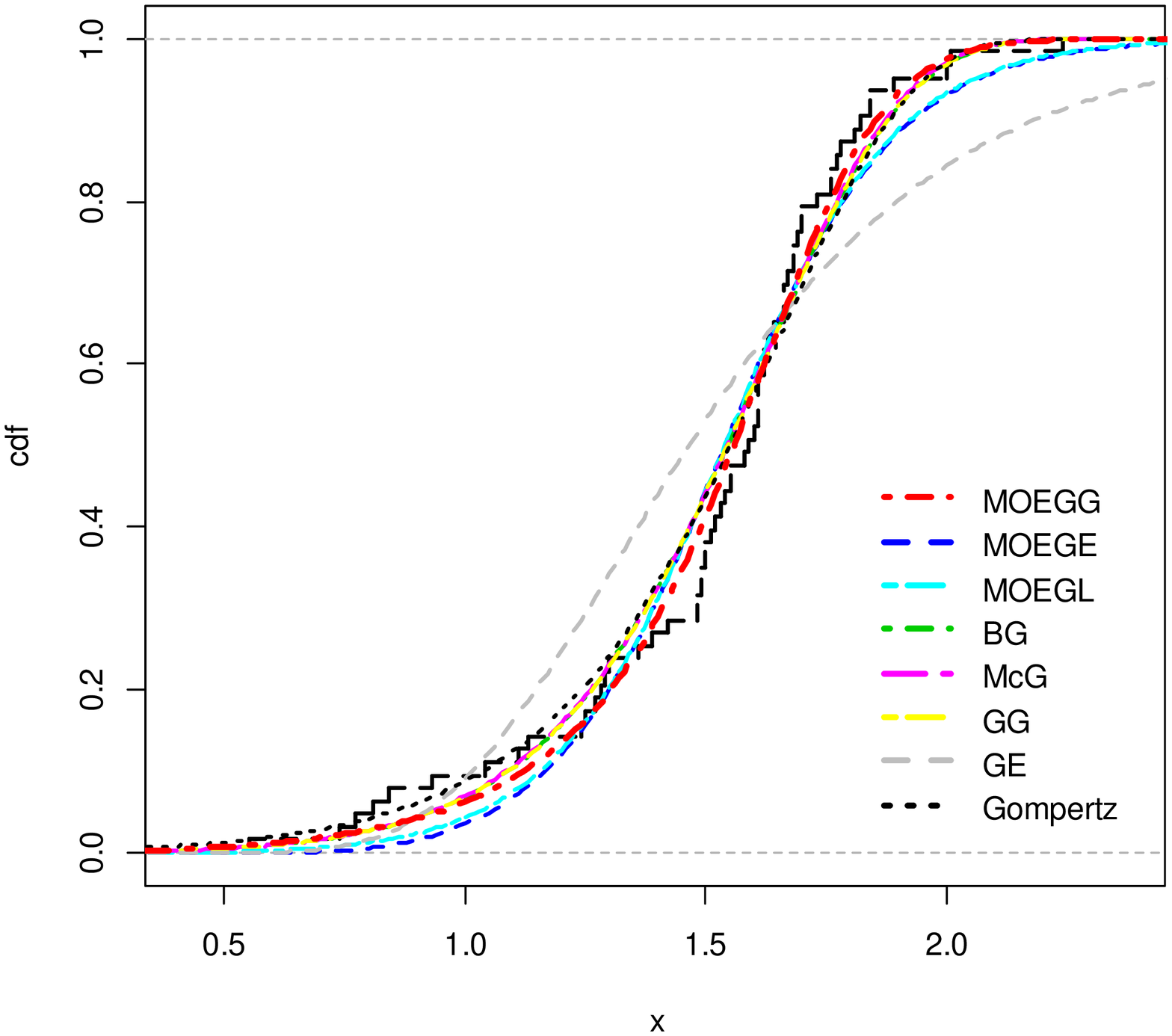}%
\\
Figure 4. The estimated cdfs and the empirical function.
\end{center}}}

\bigskip%
{\parbox[b]{5.4613in}{\begin{center}
\includegraphics[
height=4.785in,
width=5.4613in
]%
{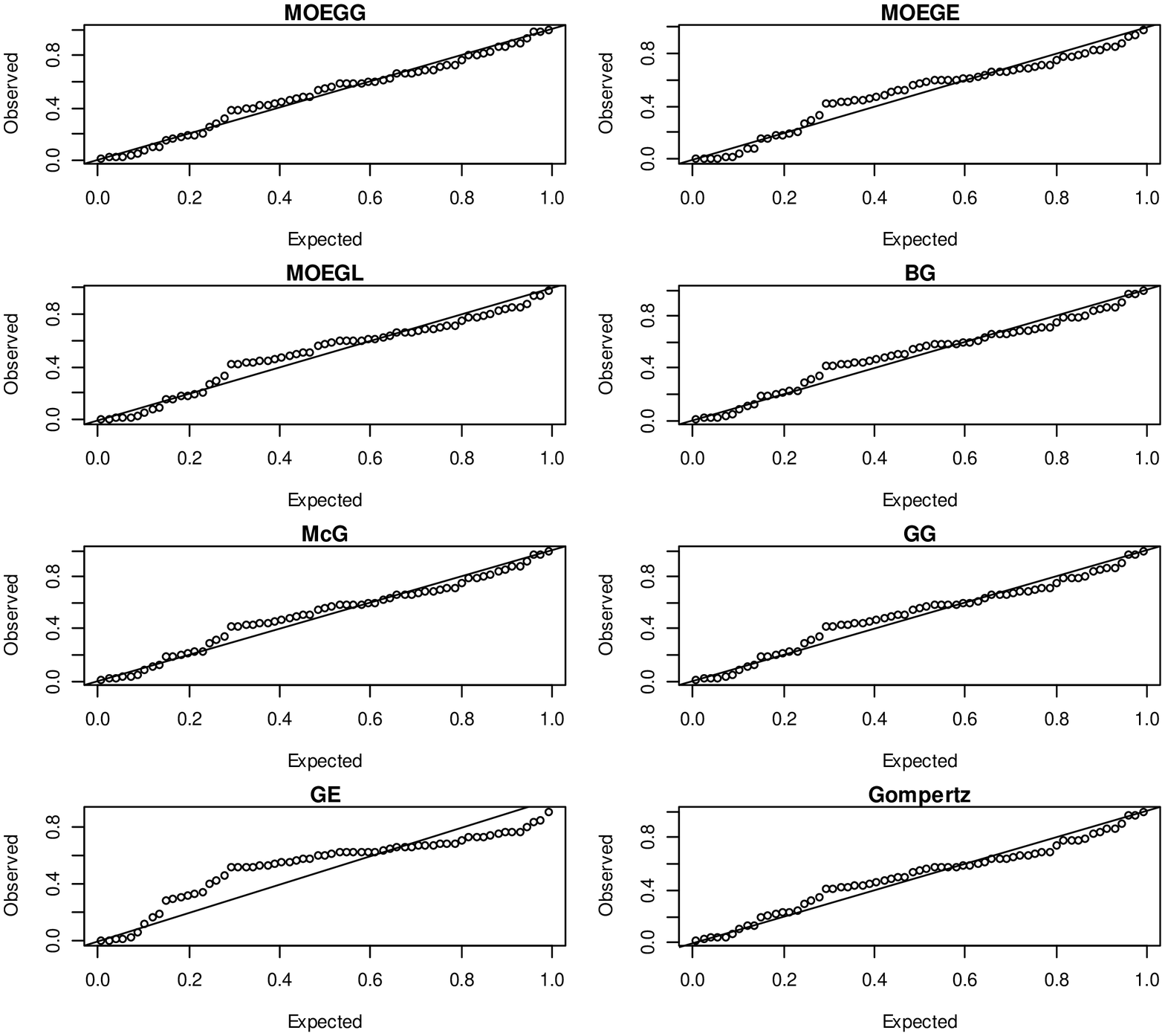}%
\\
Figure 5. The probability plots of the fitted models.
\end{center}}}

\bigskip

\section{\textbf{Conclusion\label{sec 6}}}

In this paper, we have proposed a new four-parameter continuous distribution
called the Marshall-Olkin extended generalized Gompertz distribution. This
distribution is generated by using the Marshall and Olkin method. It includes
as special sub-models the Gompertz, exponential, generalized Gompertz,
generalized exponential, Marshall-Olkin extended Gompertz, Marshall-Olkin
extended generalized exponential and Marshall-Olkin extended exponential
distributions. We have provided some mathematical properties of the proposed
model. The estimation of the unknown parameters of\textbf{\ }this model is
approached by the method of maximum likelihood and the observed information
matrix is derived. We have presented an application to real data to illustrate
the application of the proposed model. The results of this application suggest
that the new model provides consistently better fit. We hope that the new
proposed model may attract wider applications in several fields.\bigskip

\noindent\textbf{Appendix: Observed information matrix \label{appendix}%
}\medskip

\noindent Let $z_{i}=e^{-\frac{\beta}{\lambda}\left(  e^{\lambda x_{i}%
}-1\right)  }.$ The elements of the observed information matrix $J\left(
\mathbf{\Theta}\right)  $ are%
\begin{align*}
U_{\alpha\alpha}  &  =-\frac{n}{\alpha^{2}}-2\overline{\theta}\sum_{i=1}%
^{n}\frac{\left(  1-z_{i}\right)  ^{\alpha}\left(  \ln\left(  1-z_{i}\right)
\right)  ^{2}}{\theta+\overline{\theta}\left(  1-z_{i}\right)  ^{\alpha}%
}\\[0.2cm]
&  +2\overline{\theta}^{2}\sum_{i=1}^{n}\left(  \frac{\left(  1-z_{i}\right)
^{\alpha}\ln\left(  1-z_{i}\right)  }{\theta+\overline{\theta}\left(
1-z_{i}\right)  ^{\alpha}}\right)  ^{2},\\[0.6cm]
U_{\alpha\beta}  &  =\frac{1}{\lambda}\sum_{i=1}^{n}\frac{\left(  e^{\lambda
x_{i}}-1\right)  z_{i}}{1-z_{i}}-\frac{2\alpha\overline{\theta}}{\lambda}%
\sum_{i=1}^{n}\frac{\left(  e^{\lambda x_{i}}-1\right)  z_{i}\left(
1-z_{i}\right)  ^{\alpha-1}\ln\left(  1-z_{i}\right)  }{\theta+\overline
{\theta}\left(  1-z_{i}\right)  ^{\alpha}}\\[0.2cm]
&  -\frac{2\overline{\theta}}{\lambda}\sum_{i=1}^{n}\frac{\left(  e^{\lambda
x_{i}}-1\right)  z_{i}}{\left(  1-z_{i}\right)  \left[  \theta+\overline
{\theta}\left(  1-z_{i}\right)  ^{\alpha}\right]  ^{2}}\\
&  +\frac{2\alpha\overline{\theta}^{2}}{\lambda}\sum_{i=1}^{n}\frac{\left(
e^{\lambda x_{i}}-1\right)  z_{i}\left(  1-z_{i}\right)  ^{2\alpha-1}%
\ln\left(  1-z_{i}\right)  }{\left[  \theta+\overline{\theta}\left(
1-z_{i}\right)  ^{\alpha}\right]  ^{2}},\\[0.6cm]
U_{\alpha\lambda}  &  =\frac{\beta}{\lambda}\sum_{i=1}^{n}\frac{\left(
\lambda x_{i}e^{\lambda x_{i}}-e^{\lambda x_{i}}+1\right)  z_{i}}{1-z_{i}%
}\\[0.2cm]
&  -\frac{2\beta\overline{\theta}}{\lambda}\sum_{i=1}^{n}\frac{\left[
1+\alpha\ln\left(  1-z_{i}\right)  \right]  \left[  \lambda x_{i}e^{\lambda
x_{i}}-e^{\lambda x_{i}}+1\right]  z_{i}\left(  1-z_{i}\right)  ^{\alpha-1}%
}{\theta+\overline{\theta}\left(  1-z_{i}\right)  ^{\alpha}},\\[0.65cm]
U_{\alpha\theta}  &  =2\sum_{i=1}^{n}\frac{\left(  1-z_{i}\right)  ^{\alpha
}\ln\left(  1-z_{i}\right)  }{\left[  \theta+\overline{\theta}\left(
1-z_{i}\right)  ^{\alpha}\right]  ^{2}},
\end{align*}%
\begin{align*}
U_{\beta\beta}  &  =-\frac{n}{\beta^{2}}-\frac{\left(  \alpha-1\right)
}{\lambda^{2}}\sum_{i=1}^{n}\frac{\left(  e^{\lambda x_{i}}-1\right)
^{2}z_{i}}{\left(  1-z_{i}\right)  }-\frac{\left(  \alpha-1\right)  }%
{\lambda^{2}}\sum_{i=1}^{n}\frac{z_{i}^{2}\left(  e^{\lambda x_{i}}-1\right)
^{2}}{\left(  1-z_{i}\right)  ^{2}}\\
&  +\frac{2\overline{\theta}\alpha}{\lambda^{2}}\sum_{i=1}^{n}\frac{\left(
e^{\lambda x_{i}}-1\right)  ^{2}z_{i}\left(  1-z_{i}\right)  ^{\alpha-1}%
}{\theta+\overline{\theta}\left(  1-z_{i}\right)  ^{\alpha}}\\
&  -\frac{2\overline{\theta}\alpha\left(  \alpha-1\right)  }{\lambda^{2}}%
\sum_{i=1}^{n}\frac{\left(  e^{\lambda x_{i}}-1\right)  ^{2}z_{i}^{2}\left(
1-z_{i}\right)  ^{\alpha-2}}{\theta+\overline{\theta}\left(  1-z_{i}\right)
^{\alpha}}\\
&  +\frac{2\overline{\theta}^{2}\alpha^{2}}{\lambda^{2}}\sum_{i=1}^{n}\left(
\frac{\left(  e^{\lambda x_{i}}-1\right)  z_{i}\left(  1-z_{i}\right)
^{\alpha-1}}{\theta+\overline{\theta}\left(  1-z_{i}\right)  ^{\alpha}%
}\right)  ^{2},
\end{align*}

\begin{align*}
U_{\beta\lambda}  &  =\frac{1}{\lambda^{2}}\sum_{i=1}^{n}\left(  e^{\lambda
x_{i}}-1\right)  -\frac{1}{\lambda}\sum_{i=1}^{n}x_{i}e^{\lambda x_{i}}\\
&  +\frac{\left(  \alpha-1\right)  }{\lambda^{2}}\sum_{i=1}^{n}\frac{\left(
\lambda-\beta e^{\lambda x_{i}}+\beta\right)  x_{i}e^{\lambda x_{i}}z_{i}%
}{1-z_{i}}\\
&  +\frac{\beta\left(  \alpha-1\right)  }{\lambda^{3}}\sum_{i=1}^{n}%
\frac{\left(  e^{\lambda x_{i}}-1\right)  ^{2}z_{i}}{1-z_{i}}-\frac{\left(
\alpha-1\right)  }{\lambda}\sum_{i=1}^{n}\frac{\left(  e^{\lambda x_{i}%
}-1\right)  z_{i}}{1-z_{i}}\\
&  -\frac{\beta\left(  \alpha-1\right)  }{\lambda^{3}}\sum_{i=1}^{n}%
\frac{\left(  \lambda x_{i}e^{\lambda x_{i}}-e^{\lambda x_{i}}+1\right)
\left(  e^{\lambda x_{i}}-1\right)  z_{i}^{2}}{\left(  1-z_{i}\right)  ^{2}}\\
&  -\frac{2\overline{\theta}\alpha}{\lambda}\sum_{i=1}^{n}\frac{x_{i}%
e^{\lambda x_{i}}z_{i}\left(  1-z_{i}\right)  ^{\alpha-1}}{\theta
+\overline{\theta}\left(  1-z_{i}\right)  ^{\alpha}}\\
&  -\frac{2\overline{\theta}\alpha\beta}{\lambda^{3}}\sum_{i=1}^{n}%
\frac{\left(  e^{\lambda x_{i}}-\lambda x_{i}e^{\lambda x_{i}}+1\right)
\left(  e^{\lambda x_{i}}-1\right)  z_{i}\left(  1-z_{i}\right)  ^{\alpha-1}%
}{\theta+\overline{\theta}\left(  1-z_{i}\right)  ^{\alpha}}\\
&  -\frac{2\overline{\theta}\alpha\beta\left(  \alpha-1\right)  }{\lambda^{3}%
}\sum_{i=1}^{n}\frac{\left(  \lambda x_{i}e^{\lambda x_{i}}-e^{\lambda x_{i}%
}+1\right)  \left(  e^{\lambda x_{i}}-1\right)  z_{i}^{2}\left(
1-z_{i}\right)  ^{\alpha-2}}{\theta+\overline{\theta}\left(  1-z_{i}\right)
^{\alpha}}\\
&  +\frac{2\overline{\theta}\alpha}{\lambda^{2}}\sum_{i=1}^{n}\frac{\left(
e^{\lambda x_{i}}-1\right)  z_{i}\left(  1-z_{i}\right)  ^{\alpha-1}}%
{\theta+\overline{\theta}\left(  1-z_{i}\right)  ^{\alpha}}\\
&  +\frac{2\overline{\theta}^{2}\beta\alpha}{\lambda^{3}}\sum_{i=1}^{n}%
\frac{\left(  \lambda x_{i}e^{\lambda x_{i}}-e^{\lambda x_{i}}+1\right)
\left(  e^{\lambda x_{i}}-1\right)  z_{i}^{2}\left(  1-z_{i}\right)
^{2\alpha-2}}{\left[  \theta+\overline{\theta}\left(  1-z_{i}\right)
^{\alpha}\right]  ^{2}},
\end{align*}%
\begin{align*}
U_{\lambda\lambda}  &  =-\frac{2\beta}{\lambda^{3}}\sum_{i=1}^{n}\left(
e^{\lambda x_{i}}-1\right)  +\frac{2\beta}{\lambda^{2}}\sum_{i=1}^{n}%
x_{i}e^{\lambda x_{i}}-\frac{\beta}{\lambda}\sum_{i=1}^{n}x_{i}^{2}e^{\lambda
x_{i}}\\
&  +\frac{\beta\left(  \alpha-1\right)  }{\lambda}\sum_{i=1}^{n}\frac
{z_{i}x_{i}^{2}e^{\lambda x_{i}}}{1-z_{i}}\\
&  -\frac{\beta^{2}\left(  \alpha-1\right)  }{\lambda^{4}}\sum_{i=1}^{n}%
\frac{\left(  \lambda x_{i}e^{\lambda x_{i}}-e^{\lambda x_{i}}+1\right)
^{2}z_{i}}{1-z_{i}}\\
&  -\frac{2\beta\left(  \alpha-1\right)  }{\lambda^{3}}\sum_{i=1}^{n}%
\frac{\left(  \lambda x_{i}e^{\lambda x_{i}}-e^{\lambda x_{i}}+1\right)
z_{i}}{1-z_{i}}\\
&  -\frac{\beta^{2}\left(  \alpha-1\right)  }{\lambda^{4}}\sum_{i=1}%
^{n}\left(  \frac{\left(  \lambda x_{i}e^{\lambda x_{i}}-e^{\lambda x_{i}%
}+1\right)  z_{i}}{1-z_{i}}\right)  ^{2}\\
&  -2\alpha\beta\overline{\theta}\sum_{i=1}^{n}\frac{x_{i}^{2}e^{\lambda
x_{i}}z_{i}\left(  1-z_{i}\right)  ^{\alpha-1}}{\theta+\overline{\theta
}\left(  1-z_{i}\right)  ^{\alpha}}\\
&  +\frac{2\alpha\beta^{2}\overline{\theta}}{\lambda^{3}}\sum_{i=1}^{n}%
\frac{\left(  \lambda x_{i}e^{\lambda x_{i}}-e^{\lambda x_{i}}+1\right)
^{2}z_{i}\left(  1-z_{i}\right)  ^{\alpha-1}}{\theta+\overline{\theta}\left(
1-z_{i}\right)  ^{\alpha}}\\
&  -\frac{2\beta^{2}\alpha\left(  \alpha-1\right)  \overline{\theta}}%
{\lambda^{3}}\sum_{i=1}^{n}\frac{\left(  \lambda x_{i}e^{\lambda x_{i}%
}-e^{\lambda x_{i}}+1\right)  ^{2}z_{i}^{2}\left(  1-z_{i}\right)  ^{\alpha
-2}}{\theta+\overline{\theta}\left(  1-z_{i}\right)  ^{\alpha}}\\
&  +\frac{4\beta\alpha\overline{\theta}}{\lambda^{3}}\sum_{i=1}^{n}%
\frac{\left(  \lambda x_{i}e^{\lambda x_{i}}-e^{\lambda x_{i}}+1\right)
z_{i}\left(  1-z_{i}\right)  ^{\alpha-1}}{\theta+\overline{\theta}\left(
1-z_{i}\right)  ^{\alpha}}\\
&  +\frac{2\left(  \alpha\beta\overline{\theta}\right)  ^{2}}{\lambda^{4}}%
\sum_{i=1}^{n}\left(  \frac{\left(  \lambda x_{i}e^{\lambda x_{i}}-e^{\lambda
x_{i}}+1\right)  z_{i}\left(  1-z_{i}\right)  ^{\alpha-1}}{\theta
+\overline{\theta}\left(  1-z_{i}\right)  ^{\alpha}}\right)  ^{2},\\[0.8cm]
U_{\lambda\theta}  &  =\frac{2\alpha\beta}{\lambda^{2}}\sum_{i=1}^{n}%
\frac{\left(  \lambda x_{i}e^{\lambda x_{i}}-e^{\lambda x_{i}}+1\right)
z_{i}\left(  1-z_{i}\right)  ^{\alpha-1}}{\theta+\overline{\theta}\left(
1-z_{i}\right)  ^{\alpha}}\\
&  +\frac{2\alpha\beta\overline{\theta}}{\lambda}\sum_{i=1}^{n}\frac{\left(
\lambda x_{i}e^{\lambda x_{i}}-e^{\lambda x_{i}}+1\right)  z_{i}\left(
1-z_{i}\right)  ^{\alpha-1}\left[  1-\left(  1-z_{i}\right)  ^{\alpha}\right]
}{\left[  \theta+\overline{\theta}\left(  1-z_{i}\right)  ^{\alpha}\right]
^{2}},
\end{align*}

\begin{align*}
U_{\beta\theta}  &  =\frac{2\alpha}{\lambda}\sum_{i=1}^{n}\frac{\left(
e^{\lambda x_{i}}-1\right)  z_{i}\left(  1-z_{i}\right)  ^{\alpha-1}}%
{\theta+\overline{\theta}\left(  1-z_{i}\right)  ^{\alpha}}\\
&  +2\overline{\theta}\alpha\sum_{i=1}^{n}\frac{\left[  1-\left(
1-z_{i}\right)  ^{\alpha}\right]  \left(  e^{\lambda x_{i}}-1\right)
z_{i}\left(  1-z_{i}\right)  ^{\alpha-1}}{\left[  \theta+\overline{\theta
}\left(  1-z_{i}\right)  ^{\alpha}\right]  ^{2}},
\end{align*}
and%
\[
U_{\theta\theta}=-\frac{n}{\theta^{2}}+2\sum_{i=1}^{n}\left(  \frac{1-\left(
1-z_{i}\right)  ^{\alpha}}{\theta+\overline{\theta}\left(  1-z_{i}\right)
^{\alpha}}\right)  ^{2}.
\]

\end{document}